\documentclass{amsart}
\usepackage[dvips]{epsfig}
\usepackage{graphicx}
\usepackage{latexsym}
\usepackage{amsmath}
\usepackage{amsthm}
\usepackage{amssymb}
\RequirePackage[colorlinks=true,citecolor=blue,urlcolor=blue]{hyperref}

\usepackage{eucal}
\usepackage{float}
\usepackage{xcolor}
\usepackage{multirow} 

 \usepackage[applemac]{inputenc}

\newtheorem{thm}{Theorem}
\newtheorem{assumption}[thm]{Assumption}
\newtheorem{lem}[thm]{Lemma}

\newtheorem{defi}[thm]{Definition}

\newcommand{\field}[1]{\mathbb{#1}}
\newcommand{\EE}{\field{E}}
\newcommand{\E}{\field{E}}

\newcommand{\GG}{\field{G}}

\newcommand{\PP}{\field{P}}

\newcommand{\RR}{\field{R}}
\newcommand{\R}{\field{R}}
\newcommand{\TT}{\field{T}}

\newcommand{\vip}{\vskip.2cm}

\newcommand{\Bb}{{\mathcal B}}
\newcommand{\Cc}{{\mathcal C}}
\newcommand{\Dd}{\mathcal{D}}

\newcommand{\Ff}{{\mathcal F}}

\newcommand{\Pp}{{\mathcal P}}
\newcommand{\Qq}{{\mathcal Q}}
\newcommand{\Ss}{{\mathcal S}}

\newcommand{\Ttransition}{{\mathcal P}}
\newcommand{\gtilde }{\widetilde{g}}

\def \ep {\varepsilon}

\newcommand{\norme}[1]{ {\left\lVert  #1\right\rVert}}

\begin{document}

\title[Adaptive estimation for bifurcating Markov chains]{Adaptive estimation for bifurcating Markov chains}

\author{S. Val\`ere Bitseki Penda, Marc Hoffmann and Ad\'ela\"ide Olivier}

\address{S. Val\`ere Bitseki Penda, IMB, CNRS-UMR 5584,
Universit\'e de Bourgogne, 9 avenue Alain Savary, 21078 Dijon Cedex, France.}

\email{simeon-valere.bitseki-penda@u-bourgogne.fr}

\address{Marc Hoffmann, CEREMADE, CNRS-UMR 7534,
Universit\'e Paris Dauphine, Place du mar\'echal De Lattre de Tassigny
75775 Paris Cedex 16, France.}

\email{hoffmann@ceremade.dauphine.fr}

\address{Ad\'ela\"ide Olivier, CEREMADE, CNRS-UMR 7534,
Universit\'e Paris Dauphine, Place du mar\'echal De Lattre de Tassigny
75775 Paris Cedex 16, France.}

\email{olivier@ceremade.dauphine.fr}

\begin{abstract} In a first part, we prove Bernstein-type deviation inequalities for bifurcating Markov chains (BMC) under a geometric ergodicity assumption, completing former results of Guyon and Bitseki Penda, Djellout and Guillin. These preliminary results are the key ingredient to implement nonparametric wavelet thresholding estimation procedures: in a second part, we construct nonparametric estimators of the transition density of a BMC, of its mean transition density and of the corresponding invariant density, and show smoothness adaptation over various multivariate Besov classes under $L^p$-loss error, for $1 \leq p < \infty$. We prove that our estimators are (nearly) optimal in a minimax sense. As an application, we obtain new results for the estimation of the splitting size-dependent rate of growth-fragmentation models and we extend the statistical study of bifurcating autoregressive processes.
\end{abstract}

\maketitle

\textbf{Mathematics Subject Classification (2010)}: 62G05, 62M05, 60J80, 60J20, 92D25, 

\textbf{Keywords}: Bifurcating Markov chains, binary trees, deviations inequalities, nonparametric adaptive estimation, minimax rates of convergence, bifurcating autoregressive process, growth-fragmentation processes.

\section{Introduction}
\subsection{Bifurcating Markov chains}

Bifurcating Markov Chains (BMC) are Markov chains indexed by a tree (Athreya and Kang \cite{AthreyaKang98}, Benjamini and Peres \cite{Benjamini94}, Takacs \cite{Takacs01}) that are particularly well adapted to model and understand dependent data mechanisms involved in cell division. To that end, bifurcating autoregressive models (a specific class of BMC, also considered in the paper) were first introduced by Cowan and Staudte \cite{CS86}. More recently Guyon \cite{Guyon} systematically studied BMC in a general framework. In continuous time, BMC encode certain piecewise deterministic Markov processes on trees that serve as the stochastic realisation of growth-fragmentation models (see {\it e.g.} Doumic {\it et al.} \cite{DHKR1}, Robert {\it et al.} \cite{DHKR2} for modelling cell division in {\it Escherichia coli} and the references therein).\\

For $m\geq 0$, let  $\GG_m = \{ 0,1 \}^m $ (with $\GG_0=\{\emptyset\}$) and introduce the infinite genealogical tree
$$\TT = \bigcup_{m=0}^\infty \GG_m.$$
For $u \in \mathbb G_m$, set $|u|=m$ and define the concatenation $u0 = (u,0)\in \mathbb G_{m+1}$ and $u1=(u,1)\in \mathbb G_{m+1}$.
A bifurcating Markov chain is specified by {\bf 1)} a measurable state space $(\mathcal S,\mathfrak S)$ with a Markov kernel (later called $\mathbb T$-transition) $\Ttransition$ from $(\mathcal S, \mathfrak S)$ to $(\mathcal S \times \mathcal S, \mathfrak S \otimes \mathfrak S)$ and 
{\bf 2)} a filtered probability space $\big(\Omega, \mathcal F, (\mathcal F_m)_{m \geq 0}, \PP\big)$. 
Following Guyon, 
we have the 
\begin{defi} \label{def BMC}
A bifurcating Markov chain is a family $(X_u)_{u \in \mathbb T}$ of random variables with value in $(\mathcal S, \mathfrak S)$ such that 
$X_u$ is $\mathcal F_{|u|}$-measurable for every $u \in \mathbb T$ and 
$$\E\big[\prod_{u \in \mathbb G_m} g_u(X_u, X_{u0}, X_{u1})\big|\mathcal F_m\big]= \prod_{u \in \mathbb G_m}\Ttransition g_u(X_u)$$
for every $m \geq 0$ and any family of (bounded) measurable functions $(g_u)_{u \in \mathbb G_m}$, where
$\Ttransition g(x)=\int_{\mathcal S \times \mathcal S}g(x,y,z)\Ttransition(x,dy\,dz)$ denotes the action of $\Ttransition$ on $g$.
\end{defi}
The distribution of $(X_u)_{u \in \mathbb T}$ is thus entirely determined by $\Ttransition$ and an initial distribution for $X_{\emptyset}$. Informally, we may view $(X_u)_{u \in \mathbb T}$ as a population of individuals, cells or particles indexed by $\mathbb T$ and governed by the following dynamics: to each $u \in \mathbb T$ we associate a trait $X_u$ (its size, lifetime, growth rate, DNA content and so on) with value in $\mathcal S$. At its time of death, the particle $u$ gives rize to two children $u0$ and $u1$. Conditional on $X_u=x$, the trait $(X_{u0}, X_{u1}) \in \mathcal S \times \mathcal S$ of the offspring of $u$ is distributed according to $\Ttransition(x,dy\,dz)$.\\

For $n \geq 0$, let $\mathbb T_n = \bigcup_{m=0}^n \mathbb G_m$ denote the genealogical tree up to the $n$-th generation. Assume we observe $\mathbb X^n = (X_u)_{u \in \mathbb  T_n}$, {\it i.e.} we have $2^{n+1}-1$ random variables with value in $\mathcal S$. 
There are several objects of interest that we may try to infer from the data $\mathbb X^n$. Similarly to fragmentation processes (see {\it e.g.} Bertoin \cite{bertoin}) a key role for both asymptotic and non-asymptotic analysis of bifurcating Markov chains is played by the so-called {\it tagged-branch chain}, as shown by Guyon \cite{Guyon} and Bitseki Penda {\it et al.} \cite{BDG14}. The tagged-branch chain $(Y_m)_{m \geq 0}$ corresponds to a lineage picked at random in the population $(X_u)_{u \in \mathbb T}$: it is a Markov chain with value in $\mathcal S$ defined by $Y_0=X_{\emptyset}$ and for $m \geq 1$,
$$Y_{m}=X_{\emptyset \epsilon_1\cdots\epsilon_{m}},$$
where $(\epsilon_m)_{m\geq 1}$ is a sequence of independent Bernoulli variables with parameter $1/2$, independent of $(X_u)_{u \in \mathbb T}$. It has transition
\begin{equation*} 
\displaystyle \Qq = \left(\Pp_{0}+\Pp_{1}\right)/2,
\end{equation*}
obtained from the marginal transitions 
$$\Pp_0(x,dy)=\int_{z \in \mathcal S} \Ttransition(x,dy \,dz)\;\;\text{and}\;\;\Pp_1(x,dz)=\int_{y \in \mathcal S} \Ttransition(x,dy\, dz)$$
of $\Ttransition$.
Guyon proves in \cite{Guyon} that if $(Y_m)_{m \geq 0}$ is ergodic with invariant measure $\nu$, then the convergence 
\begin{equation} \label{eq:lgn1}
\frac{1}{|\GG_n|}\sum\limits_{u\in\GG_n}g(X_u) \rightarrow \int_\mathcal S g(x)\nu(dx)
\end{equation}
holds almost-surely as $n \rightarrow \infty$ for appropriate test functions $g$. Moreover, we also have convergence results of the type 
\begin{equation} \label{eq:lgn2}
\frac{1}{|\TT_{n}|}\sum\limits_{u\in\TT_{n}}g(X_{u}, X_{u0}, X_{u1}) \rightarrow \int_\mathcal S \Ttransition g(x)\nu(dx)
\end{equation}
almost-surely as $n \rightarrow \infty$. These results are appended with central limit theorems (Theorem 19 of \cite{Guyon}) and Hoeffding-type deviations inequalities in a non-asymptotic setting (Theorem 2.11 and 2.12 of Bitseki Penda {\it et al.} \cite{BDG14}). 

\subsection{Objectives} 

The observation of $\mathbb X^n$ enables us to identify $\nu(dx)$ as  $n \rightarrow \infty$ thanks to \eqref{eq:lgn1}. Consequently, convergence \eqref{eq:lgn2} reveals $\Ttransition$ and therefore $\Qq$ is identified as well, at least asymptotically. The purpose of the present paper is at least threefold:
\begin{itemize}
\item[\bf 1)] Construct -- under appropriate regularity conditions -- estimators of $\nu, \Qq$ and $\Ttransition$ and study their rates of convergence as $n \rightarrow \infty$ under various loss functions. 
When $\mathcal S \subseteq \R$ and when $\Ttransition$ is absolutely continuous w.r.t. the Lebesgue measure, we estimate the corresponding density functions under various smoothness class assumptions and build {\it smoothness adaptive} estimators, {\it i.e.} estimator that achieve an optimal rate of convergence without prior knowledge of the smoothness class.\\ 

\item[\bf 2)]  Apply these constructions to investigate further specific classes of BMC. These include binary growth-fragmentation processes, where we subsequently estimate adaptively the splitting rate of a size-dependent model, thus extending previous results of Doumic {\it et al.} \cite{DHKR1} and bifurcating autoregressive processes, where we complete previous studies of Bitseki Penda {\it et al.} \cite{BPEBG15} and Bitseki Penda and Olivier \cite{BO}.\\

\item[\bf 3)] For the estimation of $\nu, \Qq$ and $\Ttransition$ and the subsequent estimation results of {\bf 2)}, prove that our results are sharp in a minimax sense.\\
\end{itemize}
Our smoothness adaptive estimators are based on wavelet thresholding for density estimation (Donoho {\it et al.} \cite{DJKP2} in the generalised framework of Kerkyacharian and Picard \cite{KP}). Implementing these techniques requires concentration properties of empirical wavelet coefficients. To that end, we prove new deviation inequalities for bifurcating Markov chains that we develop independently in a more general setting, when $\mathcal S$ is not necessarily restricted to $\R$. Note also that when $\Pp_0=\Pp_1$, we have $\Qq=\Pp_0=\Pp_1$ as well and we retrieve the usual framework of nonparametric estimation of Markov chains when the observation is based on $(Y_i)_{1 \leq i \leq n}$ solely. We are therefore in the line of combining and generalising the study of Cl\'emen\c{c}on \cite{Clemencon2} and Lacour \cite{Lacour1, Lacour2} that both consider adaptive estimation for Markov chains when $\mathcal S \subseteq \R$. 

\subsection{Main results and organisation of the paper}

In Section~\ref{sec:deviations}, we generalise the Hoeffding-type deviations inequalities of Bitseki Penda {\it et al.} \cite{BDG14} for BMC to Bernstein-type inequalities: when $\Ttransition$ is uniformly geometrically ergodic (Assumption~\ref{ass:ergq} below), we prove in Theorem~\ref{thm:triplets} deviations of the form
$$
\PP \Big( \frac{1}{|\GG_n|} \sum_{u \in \GG_n} g(X_u,X_{u0},X_{u1}) - \int \Ttransition g\, d\nu \geq \delta \Big) \leq \exp\Big(-\frac{\kappa |\GG_n| \delta^2}{\Sigma_n(g) + |g|_\infty \delta} \Big)
$$
and
$$
\PP \Big( \frac{1}{|\TT_n|} \sum_{u \in \TT_n} g(X_u,X_{u0},X_{u1}) - \int \Ttransition g\, d\nu \geq \delta \Big) \leq \exp\Big(-\frac{\tilde \kappa n^{-1}|\TT_n| \delta^2}{\Sigma_n(g) + |g|_\infty \delta} \Big),
$$
where $\kappa, \tilde \kappa>0$ only depend on $\Ttransition$
and $\Sigma_n(g)$ is a variance term which depends on a combination of the $L^p$-norms of $g$ for $p = 1,2,\infty$ w.r.t. a common dominating measure for the family $\{\Qq(x,dy), x \in \mathcal S\}$.  The precise results are stated in Theorems~\ref{thm:1step} and~\ref{thm:triplets}.\\ 

Section~\ref{sec:stat} is devoted to the statistical estimation of $\nu, \Qq$ and $\Ttransition$ when $\mathcal S \subseteq \R$ and the family $\{\Ttransition(x,dy\,dz), x \in \mathcal S\}$ is dominated by the Lebesgue measure on $\R^2$. In that setting, abusing notation slightly, we have $\nu(dx)=\nu(x)dx$, $\Qq(x,dy)=\Qq(x,y)dy$ and $\Ttransition(x,dy\,dz)=\Ttransition(x,y,z)dydz$ for some functions $x \leadsto \nu(x)$, $(x,y)\leadsto \Qq(x,y)$ and $(x,y,z)\leadsto \Ttransition(x,y,z)$ that we reconstruct nonparametrically. Our estimators are constructed in several steps:
\begin{itemize}
\item[{\bf i)}] We approximate the functions $\nu(x)$, $f_\Qq(x,y)=\nu(x)\Qq(x,y)$ and $f_\Ttransition(x,y,z)=\nu(x) \Ttransition(x,y,z)$ by atomic representations 
\begin{align*}
\nu(x) &\approx \sum_{\lambda \in \mathcal V^1(\nu)} \langle \nu,\psi_\lambda^1\rangle \psi_\lambda^1(x),\\
f_\Qq(x,y) & \approx \sum_{\lambda \in \mathcal V^2(f_\Qq)} \langle f_\Qq,\psi_\lambda^2\rangle \psi_\lambda^2(x,y),\\
f_\Ttransition(x,y,z) & \approx \sum_{\lambda \in \mathcal V^3(f_\Ttransition)} \langle f_\Ttransition,\psi_\lambda^3\rangle \psi_\lambda^3(x,y,z),
 \end{align*}
where $\langle \cdot, \cdot \rangle$ denotes the usual $L^2$-inner product (over $\R^d$, for $d=1,2,3$ respectively) and $\big(\psi_\lambda^d, \lambda \in \mathcal V^d(\cdot)\big)$ is a collection of functions (wavelets) in $L^2(\R^d)$ that are localised in time and frequency, indexed by a set $\mathcal V^d(\cdot)$ that depends on the signal itself\footnote{The precise meaning of the symbol $\approx$ and the properties of the $\psi_\lambda$'s are stated precisely in Section~\ref{atomic decomposition}.}.\\

\item[{\bf ii)}]  We estimate 
\begin{align*}
\langle  \nu,\psi_\lambda^1\rangle &\;\; \text{by} \;\;|\mathbb T_n|^{-1}\sum_{u \in \mathbb T_n}\psi_\lambda^1(X_u),\\
\langle f_\Qq,\psi_\lambda^2\rangle &\;\; \text{by}\;\; |\mathbb T_n^\star|^{-1}\sum_{u \in \mathbb T_n^\star}\psi_\lambda^2(X_{u^-}, X_u),\\
\langle f_\Ttransition,\psi_\lambda^3\rangle &\;\; \text{by} \;\;|\mathbb T_{n-1}|^{-1}\sum_{u \in \mathbb T_{n-1}}\psi_\lambda^3(X_u, X_{u0}, X_{u1}),
\end{align*}
where  $X_{u^-}$ denotes the trait of the parent of $u$ and $\mathbb T_n^\star = \mathbb T_n \setminus \mathbb G_0$, and specify a selection rule for $\mathcal V^d(\cdot)$ (with the dependence in the unknown function somehow replaced by an estimator). The rule is dictated by hard thresholding over the estimation of the coefficients that are kept only if they exceed some noise level, tuned with $|\mathbb T_n|$ and prior knowledge on the unknown function, as follows by standard density estimation by wavelet thresholding (Donoho {\it et al.} \cite{DJKP}, Kerkyacharian and Picard \cite{KP}).\\ 

\item[{\bf iii)}]  Denoting by $\widehat \nu_n(x)$, $\widehat f_{n}(x,y)$ and $\widehat f_{n}(x,y,z)$ the estimators of $\nu(x)$, $f_\Qq(x,y)$ and $f_\Ttransition(x,y,z)$ respectively constructed in Step {\bf ii)}, we finally take as estimators for $\Qq(x,y)$ and $\Ttransition(x,y,z)$ the quotient estimators 
$$\widehat \Qq_n(x,y)=\frac{\widehat f_{n}(x,y)}{\widehat \nu_n(x)}\;\;\text{and}\;\;\widehat {\Ttransition}_n(x,y,z)=\frac{\widehat f_{n}(x,y,z)}{\widehat \nu_n(x)}$$
provided $\widehat \nu_n(x)$ exceeds a minimal threshold.
\end{itemize}
Beyond the inherent technical difficulties of the approximation Steps i) and iii), the crucial novel part is the estimation Step ii), where Theorems~\ref{thm:1step} and~\ref{thm:triplets} are used to estimate precisely the probability that the thresholding rule applied to the empirical wavelet coefficient is close in effect to thresholding the true coefficients.

When $\nu, \Qq$ or $\Ttransition$ (identified with their densities w.r.t. appropriate dominating measures) belong to an isotropic Besov ball of smoothness $s$ measured in $L^\pi$ over a domain $\mathcal D^d$ in $\R^d$, with $s > d/\pi$ and $d=1,2,3$ respectively, we prove in Theorems~\ref{thm:ratenu},~\ref{thm:rateq} and~\ref{thm:rateP} that if $\mathcal Q$ is uniformly geometrically ergodic, then our estimators achieve the rate $|\mathbb T_n|^{-\alpha_d(s,p,\pi)}$ in $L^p(\mathcal D)$-loss, up to additional $\log |\mathbb T_n|$ terms, where 
$$\alpha_d(s,p,\pi) =\min \Big\{ \frac{s}{2s+d}, \frac{s+d(1/p-1/\pi)}{2s+d(1-2/\pi)}\Big\}$$ 
is the usual exponent for the minimax rate of estimation of a $d$-variate function with order of smoothness $s$ measured in $L^\pi$ in $L^p$-loss error. This rate is nearly optimal in a minimax sense for $d=1$, as follows from particular case $\Qq(x,dy)=\nu(dy)$ that boils down to density estimation with  $|\mathbb T_n|$ data: the optimality is then a direct consequence of Theorem 2 in Donoho {\it et al.} \cite{DJKP}. As for the case $d=2$ and $d=3$, the structure of BMC comes into play and we need to prove a specific optimality result, stated in Theorems~\ref{thm:rateq} and~\ref{thm:rateP}. We rely on classical lower bound techniques for density estimation and Markov chains (Hoffmann \cite{H4}, Cl\'emen\c{c}on \cite{Clemencon2}, Lacour \cite{Lacour1, Lacour2}).\\

We apply our generic results in Section~\ref{sec:examples} to two illustrative examples. We consider in Section~\ref{application DHKR1} the growth-fragmentation model as studied in Doumic {\it et al.} \cite{DHKR1}, where we estimate the size-dependent splitting rate of the model as a function of the invariant measure of an associated BMC in Theorem~\ref{prop:rateB}. This enables us to extend the recent results of Doumic {\it et al.} in several directions: adaptive estimation, extension of the smoothness classes and the loss functions considered, and also a proof of a minimax lower bound. In Section~\ref{application BAR}, we show how bifurcating autoregressive models (BAR) as developped for instance in de Saporta {\it et al.} \cite{BdSGP09} and Bitseki Penda and Olivier \cite{BO} are embedded into our generic framework of estimation. A numerical illustration highlights the feasibility of our procedure in practice and is presented in Section~\ref{numerical}. The proofs are postponed to Section~\ref{sec:proofs}.

\section{Deviations inequalities for empirical means} \label{sec:deviations}
In the sequel, we fix a (measurable) subset $\mathcal D \subseteq \mathcal S$ that will be later needed for statistical purposes.
We need some regularity on the $\TT$-transition $\Ttransition$ via its mean transition $\Qq=\tfrac{1}{2}(\Pp_0+\Pp_1)$. 

\begin{assumption} \label{ass:densityq}
The family $\{\Qq(x,dy), x\in \mathcal S\}$ is dominated by a common sigma-finite measure $\mathfrak n(dy)$. We have (abusing notation slightly)
$$\Qq(x,dy)=\Qq(x,y)\mathfrak n(dy)\;\;\text{for every}\;\;x\in \mathcal S,$$
for some $\Qq:\mathcal S^2\rightarrow [0,\infty)$ such that 
$$|\Qq|_{\mathcal D} = \sup_{x \in \mathcal S, y \in \mathcal D}\Qq(x,y) <\infty.$$
\end{assumption}
An invariant probability measure for $\Qq$ is a probability $\nu$ on $(\mathcal S, \mathfrak S)$ such that
$\nu \Qq=\nu$
where $\nu \Qq(dy) = \int_{x \in \mathcal S}\nu(dx)\Qq(x,dy)$. 
We set
$$\Qq^{r}(x,dy)=\int_{z\in \Ss}\Qq(x,dz)\Qq^{r-1}(z,dy)\;\;\text{with}\;\;\Qq^0(x,dy) = \delta_x(dy)$$
for the $r$-th iteration of $\Qq$. 
For a function  $g :\Ss^d \rightarrow \RR$ with $d=1,2,3$ and $1 \leq p \leq \infty$, we denote by $|g|_p$ its $L^p$-norm w.r.t. the measure $\mathfrak n^{\otimes d}$, allowing for the value $|g|_p=\infty$ if $g \notin L^p(\mathfrak n^{\otimes d})$. The same notation applies to a function  $g:\mathcal D^d \rightarrow \R$ tacitly considered as a function from $\mathcal S^d\rightarrow \R$ by setting $g(x)=0$ for $x \in \mathcal S \setminus \mathcal D$.
\begin{assumption}\label{ass:ergq} The mean transition $\Qq$ admits a unique invariant probability measure $\nu$
and there exist $R>0$ and $0 < \rho <1/2$ such that
\begin{equation*}\label{eq:lyapunov}
\big|\Qq^{m}g(x) - \int_{\mathcal S} g\,d\nu\big|\leq R | g |_{\infty} \, \rho^{m}, \quad x \in \Ss , \quad m \geq 0,
\end{equation*}
for every $g$ integrable w.r.t. $\nu$. 
\end{assumption}
Assumption~\ref{ass:ergq} is a uniform geometric ergodicity condition that can be verified in most applications using the theory of Meyn and Tweedie \cite{MT}. 
The ergodicity rate should be small enough ($\rho < 1/2$) and this point is crucial for the proofs. However this is sometimes delicate to check in applications and we refer to Hairer and Mattingly \cite{HairerMattingly} for an explicit control of the ergodicity rate.\\

Our first result is a deviation inequality for empirical means over $\mathbb G_n$ or $\mathbb T_n$. We need some notation. Let
\begin{align*}
\kappa_1 = & \kappa_1(\Qq, \mathcal D) = 32 \max \big\{ |\Qq|_{\mathcal D} , 4 |\Qq|_{\mathcal D}^2,  4 R^2(1+\rho)^2 \big\}, \\
\kappa_2 = & \kappa_2(\Qq) = \tfrac{16}{3}  \max\big\{ 1+R\rho ,  R(1+\rho)\big\}, \\
\kappa_3 = & \kappa_3(\Qq, \mathcal D) = 96 \max \big\{ |\Qq|_{\mathcal D} , 16 |\Qq|_{\mathcal D}^2 , 4 R^2(1+\rho)^2 (1-2\rho)^{-2}  \big\}, \\
\kappa_4 = & \kappa_4(\Qq) = \tfrac{16}{3}  \max \big\{ 1+R\rho ,  R(1+\rho)(1-2\rho)^{-1} \big\}, 
\end{align*}
where $|\Qq|_{\mathcal D} = \sup_{x \in \mathcal S, y \in \mathcal D}\Qq(x,y)$ is defined in Assumption~\ref{ass:densityq}. For $g:\mathcal S^d \rightarrow \R$, define $\Sigma_{1,1}(g)=|g|_2^2$ and for $n \geq 2$,
\begin{equation} \label{def sigma 1n}
\Sigma_{1,n}(g)  = |g|_2^2 + \min_{1 \leq \ell \leq n-1}\big(|g|_1^2 2^\ell+ |g|_\infty^2 2^{-\ell} \big).
\end{equation}
Define also $\Sigma_{2,1}(g) =  |\Ttransition g^2|_{1}$ and for $n\geq 2$,
\begin{equation} \label{eq:sigma2}
\Sigma_{2,n}(g)  =  |\Ttransition g^2|_{1} + \min_{1 \leq \ell \leq n-1}\big(|\Ttransition g|_1^2 2^\ell + |\Ttransition g|_\infty^2 2^{-\ell} \big).
\end{equation}
\begin{thm}
\label{thm:1step}
Work under Assumptions~\ref{ass:densityq} and~\ref{ass:ergq}. Then, for every $n\geq 1$ and every $g: \mathcal D \subseteq \mathcal S \rightarrow \R$ integrable w.r.t. $\nu$, the following inequalities hold true:\\

\noindent {\bf (i)} For any $\delta > 0$ such that $\delta \geq 4 R |g|_\infty |\GG_n|^{-1}$, we have
$$
\PP \Big( \frac{1}{|\GG_n|} \sum_{u \in \GG_n} g(X_u) - \int_{\mathcal S} g\,d\nu \geq \delta \Big) \leq  \exp\Big(\frac{ - |\GG_n| \delta^2}{ \kappa_1 \Sigma_{1,n}(g) + \kappa_2 |g|_\infty   \delta} \Big).
$$
\noindent {\bf (ii)} For any $\delta > 0$ such that $\delta \geq 4R (1-2\rho)^{-1} |g|_\infty |\TT_n|^{-1}$, we have
$$
\PP \Big( \frac{1}{|\TT_n|} \sum_{u \in \TT_n} g(X_u) - \int_\mathcal S g\,d\nu\geq \delta \Big) \leq \exp\Big(\frac{ - |\TT_n| \delta^2}{ \kappa_3 \Sigma_{1,n}(g) + \kappa_4 |g|_\infty \delta } \Big).
$$

\end{thm}

\vip
\begin{thm}
\label{thm:triplets}
Work under Assumptions~\ref{ass:densityq} and~\ref{ass:ergq}.
Then, for every $n\geq2$ and for every $g:\mathcal D^3 \subseteq \mathcal S^3 \rightarrow \R$ such that $\Ttransition g$ is well defined and integrable w.r.t. $\nu$, the following inequalities hold true:\\

\noindent {\bf (i)} For any $\delta > 0$ such that $\delta \geq 4R |\Ttransition g|_\infty |\GG_n|^{-1}$, we have
$$
\PP \Big( \frac{1}{|\GG_n|} \sum_{u \in \GG_n} g(X_u,X_{u0},X_{u1}) - \int_\mathcal S\Ttransition g\,d\nu \geq \delta \Big) \leq  \exp\Big(\frac{ - |\GG_n| \delta^2}{ \kappa_1 \Sigma_{2,n}(g) + \kappa_2 |g|_\infty   \delta} \Big).
$$
\noindent {\bf (ii)} For any $\delta > 0$ such that $\delta \geq 4(nR |\Ttransition g|_\infty + |g|_\infty) |\TT_{n-1}|^{-1}$, we have
$$
\PP \Big( \frac{1}{|\TT_{n-1}|} \sum_{u \in \TT_{n-1}} g(X_u,X_{u0},X_{u1}) - \int_\mathcal S \Ttransition g\,d\nu \geq \delta \Big) \leq \exp\Big(\frac{ - n^{-1} |\TT_{n-1}| \delta^2}{ \kappa_1 \Sigma_{2,n-1}(g) + \kappa_2  |g|_\infty \delta } \Big).
$$
\end{thm}

\vip
A few remarks are in order: \\

\noindent {\bf 1)} Theorem~\ref{thm:1step}\,(i) is a direct consequence of Theorem~\ref{thm:triplets}\,(i) but Theorem~\ref{thm:1step}\,(ii) is not a corollary of Theorem~\ref{thm:triplets}\,(ii): we note that a slow term or order $n^{-1} \approx (\log |\mathbb T_n|)^{-1}$ comes in Theorem~\ref{thm:triplets}\,(ii).

\noindent {\bf 2)}
Bitseki-Penda \textit{et al.} in \cite{BDG14} study similar Hoeffding-type deviations inequalities for functionals of bifurcating Markov chains under ergodicity assumption and for uniformly bounded functions. In the present work and for statistical purposes, we need Bernstein-type deviations inequalities which require a specific treatment than cannot be obtained from a  direct adaptation of \cite{BDG14}. In particular, we apply our results to multivariate wavelets test functions $\psi_\lambda^d$ that are well localised but unbounded, and a fine control of the conditional variance $\Sigma_{i,n}(\psi_\lambda^d)$, $i=1,2$ is of crucial importance.

\noindent {\bf 3)} Assumption~\ref{ass:ergq} about the uniform geometric ergodicity is quite strong, although  satisfied in the two examples developed in  Section~\ref{sec:examples} (at the cost however of assuming that the splitting rate of the growth-fragmentation model has bounded support in Section~\ref{application DHKR1}). Presumably, a way to relax this restriction would be to require a weaker geometric ergodicity condition of the form
$$
\big|\Qq^{m}g(x) - \int_{\mathcal S} g\,d\nu\big|\leq R | g |_{\infty}V(x) \, \rho^{m}, \quad x \in \Ss , \quad m \geq 0,
$$
for some Lyapunov function $V:\mathcal S\rightarrow [1,\infty)$. Analogous results could then be obtained via transportation information inequalities for bifurcating Markov chains with a similar approach as in Gao {\it et al.} \cite{GaoGuillinWu}, but this lies beyond the scope of the paper.

\section{Statistical estimation} \label{sec:stat}

In this section, we take $(\mathcal S, \mathfrak S)=\big(\R, \mathcal B(\R)\big)$. As in the previous section, we fix a compact interval $\mathcal D \subseteq \mathcal S$. The following assumption will be needed here

\begin{assumption} \label{densityTtrans} The family $\{\Ttransition(x,dy\,dz), x \in \mathcal S\}$ is dominated w.r.t. the Lebesgue measure on $\big(\R^2, \mathcal B(\R^2)\big)$. We have (abusing notation slightly)
$$\Ttransition(x,dy\,dz)=\Ttransition(x,y,z)dy\,dz\;\;\text{for every}\;\;x \in \mathcal S$$
for some $\Ttransition:\mathcal S^3\rightarrow [0,\infty)$ such that 
$$|\Ttransition|_{\mathcal D,1}=\int_{\mathcal S^2}\sup_{x \in \mathcal D}\Ttransition(x,y,z)dydz<\infty.$$
\end{assumption}

Under Assumptions~\ref{ass:densityq},~\ref{ass:ergq} and~\ref{densityTtrans} with $\mathfrak n(dy)=dy$, we have (abusing notation slightly)
$$\Ttransition(x,dy\,dz)=\Ttransition(x,y,z)dy\,dz,\;\;\Qq(x,dy)=\Qq(x,y)dy\;\;\text{and}\;\;\nu(dx)=\nu(x)dx.$$
For some $n \geq 1$, we observe $\mathbb X_n = (X_u)_{u \in \TT_n}$ 
and we aim at constructing nonparametric estimators of $x \leadsto \nu(x)$, $(x,y) \leadsto \Qq(x,y)$ and $(x,y,z) \leadsto \Ttransition(x,y,z)$ for $x,y,z \in \mathcal D$. To that end, we use regular wavelet bases adapted to the domain $\mathcal D^d$ for $d=1,2,3$.

\subsection{Atomic decompositions and wavelets} \label{atomic decomposition}

Wavelet bases $(\psi_\lambda^d)_\lambda$ adapted to a domain ${\mathcal D}^d$ in $\R^d$, for $d=1,2,3$ are documented in numerous textbooks, see {\it e.g.} Cohen \cite{Co}. The multi-index $\lambda$ concatenates the spatial index and the
resolution level $j=|\lambda|$. We set
$\Lambda_j=\{\lambda,\;|\lambda|=j\}$ and $\Lambda=\cup_{j \geq
-1} \Lambda_j$. Thus, for $g \in L^\pi(\mathcal D^d)$ for some $\pi \in (0,\infty]$, we have
$$g = \sum_{j \geq -1} \sum_{\lambda \in \Lambda_j}
g_\lambda \psi_\lambda^d=\sum_{\lambda \in \Lambda}g_\lambda \psi_\lambda^d,\;\;\;\text{with}\;\;g_\lambda=\langle g,\psi_\lambda^d\rangle,$$
where we have set $j=-1$ in order to incorporate the low frequency
part of the decomposition and $\langle g,\psi_\lambda^d\rangle = \int g\psi_\lambda^d$ denotes the inner product in $L^2(\R^d)$. From now on, the basis
$(\psi_{\lambda}^d)_\lambda$ is fixed. 
For $s>0$ and $\pi\in (0,\infty]$, $g$ belongs to $B_{\pi,\infty}^s({\mathcal D})$ if the following norm is finite:
\begin{equation} \label{besovanalysis}
\|g\|_{B^s_{\pi,\infty}({\mathcal D})}= \sup_{j \geq
-1}2^{j(s+d(1/2-1/\pi))}\big(\sum_{\lambda\in
\Lambda_j}|\langle g, \psi_\lambda^d\rangle|^\pi\big)^{1/\pi}
\end{equation}
with the usual modification if $\pi=\infty$. Precise connection between this definition of Besov norm and more standard ones can be found in \cite{Co}. Given a basis $(\psi_\lambda^d)_\lambda$, there exists $\sigma>0$ such that for $\pi\ge1$ and $s\le \sigma$ the Besov space defined by \eqref{besovanalysis} exactly matches the
usual definition in terms of moduli of smoothness for~$g$. The index $\sigma$ can be taken arbitrarily large. The additional properties of the wavelet basis
$(\psi_\lambda^d)_\lambda$ that we need are summarized in the next
assumption. 
\begin{assumption}  \label{AssumptionA}
For $p \geq 1$, 
\begin{equation} \label{scaling psi}
\|\psi_\lambda^d\|_{L^p}^p
\sim 2^{|\lambda|d(p/2-1)},
\end{equation}
for some $\sigma >0$ and for all $s \le \sigma$, $j_0  \ge 0$,
\begin{equation}\label{E:std_approx_besov}
\|g-\sum_{j \le j_0} \sum_{\lambda\in \Lambda_j}
g_{\lambda}\psi_{\lambda}^d\|_{L^p} \lesssim 2^{-{j_0} s}
\norme{g}_{B_{p,\infty}^s(\Dd)},
\end{equation}
for any subset $\Lambda_0 \subset \Lambda$,
\begin{equation} \label{superconcentration}
\int_{{\mathcal D}}\big(\sum_{\lambda \in \Lambda_0}|\psi_{\lambda}^d(x)|^2\big)^{p/2}dx \sim \sum_{\lambda \in \Lambda_0} \|\psi_\lambda^d\|_{L^p}^p.
\end{equation}
If $p > 1$, for any sequence $(u_\lambda)_{\lambda \in \Lambda}$,
\begin{equation} \label{unconditional}
 \big\|\big(\sum_{\lambda \in \Lambda}|u_\lambda \psi_{\lambda}^d|^2\big)^{1/2}\big\|_{L^p} \sim \|\sum_{\lambda \in \Lambda} u_\lambda \psi_\lambda^d\|_{L^p}.
\end{equation}
\end{assumption}
\noindent The symbol $\sim$ means inequality in both ways, up to a
constant depending on $p$ and ${\mathcal D}$ only. The property
\eqref{E:std_approx_besov} reflects that our definition
\eqref{besovanalysis} of Besov spaces matches the definition in term
of linear approximation. Property \eqref{unconditional} reflects an
unconditional basis property, see Kerkyacharian and Picard \cite{KP}, De Vore {\it et al.} \cite{DKT} and
\eqref{superconcentration} is referred to as a superconcentration
inequality, or Temlyakov property \cite{KP}. The formulation of \eqref{superconcentration}-\eqref{unconditional} in the context of statistical
estimation is posterior to the original papers of Donoho and
Johnstone \cite{DJ, DJ2} and Donoho {\it et al.} \cite{DJKP, DJKP2}
and is due to Kerkyacharian and Picard \cite{KP}.
The existence of compactly supported wavelet bases satisfying Assumption~\ref{AssumptionA} is discussed in Meyer
\cite{Meyer}, see also Cohen \cite{Co}.

\subsection{Estimation of the invariant density $\nu$} \label{sec:stat_nu}
Recall that we estimate $x \leadsto \nu(x)$ for $x\in \mathcal D$, taken as a compact interval in $\mathcal S \subseteq \R$. 
We approximate the representation 
$$\nu(x) = \sum_{\lambda \in \Lambda}\nu_\lambda \psi_\lambda^1(x),\;\;\nu_\lambda = \langle \nu,\psi_\lambda^1\rangle$$
by
$$\widehat \nu_n(x) = \sum_{|\lambda| \leq J} \widehat \nu_{\lambda, n}\psi_\lambda^1(x),$$
with
$$\widehat \nu_{\lambda, n} =  
\mathcal T_{\lambda,\eta}\Big(\frac{1}{|\TT_n|} \sum_{u\in\TT_n} \psi_\lambda^1(X_u)\Big),
$$
and 
$\mathcal T_{\lambda,\eta}(x) = x {\bf 1}_{|x| \geq \eta}$ denotes the standard threshold operator (with $\mathcal T_{\lambda,\eta}(x)=x$ for the low frequency part when $\lambda \in \Lambda_{-1}$). Thus $\widehat \nu_n$ is specified by the maximal resolution level $J$ and the threshold $\eta$.
\begin{thm}\label{thm:ratenu} 
Work under Assumptions~\ref{ass:densityq} and ~\ref{ass:ergq} with $\mathfrak n(dx)=dx$. Specify $\widehat \nu_n$ with $$J= \log_2\frac{|\mathbb T_n|}{\log |\mathbb T_n|}\;\; {\rm and} \;\; \eta = c\sqrt{\log |\mathbb T_n|/|\mathbb T_n|}$$ for some $c>0$.
For every $\pi \in (0,\infty], s >1/\pi$ and $p \geq 1$, for large enough $n$ and $c$, the following estimate holds
$$
\Big(\EE \big[ \|  \widehat{\nu}_n - \nu \|_{L^p(\Dd)}^p \big]\Big)^{1/p} \lesssim \Big(\frac{\log |\mathbb T_n|}{|\mathbb T_n|}\Big)^{\alpha_1(s,p,\pi)},$$
with $\alpha_1(s,p,\pi) = \min\big\{ \tfrac{s}{2s+1} , \tfrac{s+1/p-1/\pi}{2s+1-2/\pi}\big\}$,  up to a constant that depends on $s,p,\pi, \|\nu\|_{\Bb_{\pi,\infty}^{s}(\mathcal D)}$, $\rho$, $R$ and $|\Qq|_\mathcal D$ and that is continuous in its arguments.
\end{thm}

Two remarks are in order:\\

\noindent {\bf 1)} The upper-rate of convergence is the classical minimax rate in density estimation. We infer that our estimator is nearly optimal in a minimax sense as follows from Theorem 2 in Donoho {\it et al.} \cite{DJKP} applied to the class $\Qq(x,y)dy=\nu(y)dy$, {\it i.e.} in the particular case when we have i.i.d. $X_u$'s. We highlight the fact that $n$ represents here the number of observed generations in the tree, which means that we observe $|\TT_n| = 2^{n+1}-1$ traits.

\noindent {\bf 2)} The estimator $\widehat \nu_n$ is {\it smooth-adaptive} in the following sense: for every $s_0>0$, $0 < \rho_0 < 1/2$, $R_0>0$ and $\Qq_0>0$, define the sets
$
\mathcal A(s_0)  = \{(s, \pi), s\geq s_0, s_0 \geq 1/\pi\}$
and
$$\Qq(\rho_0, R_0,\Qq_0) = \{\Qq\;\text{such that}\; \rho \leq \rho_0, R\leq R_0,|\Qq|_\mathcal D, \leq \Qq_0\},$$
where $\Qq$ is taken among mean transitions for which Assumption~\ref{ass:ergq} holds. Then,
for every $C>0$, there exists $c^\star = c^\star(\mathcal D, p, s_0, \rho_0, R_0, \Qq_0, C)$ such that $\widehat \nu_n$ specified with $c^\star$ satisfies
$$\sup_n \sup_{(s, \pi) \in \mathcal A(s_0)}\sup_{\nu,\Qq}
 \Big(\frac{|\mathbb T_n|}{\log |\mathbb T_n|}\Big)^{p\alpha_1(s,p,\pi)}\EE \big[ \|  \widehat{\nu}_n - \nu \|_{L^p(\Dd)}^p \big] <\infty$$
where the supremum is taken among $(\nu,\Qq)$ such that $\nu \Qq=\nu$ with $\Qq \in \Qq(\rho_0, R_0,\Qq_0)$ and $\|\nu\|_{\Bb_{\pi,\infty}^{s}(\mathcal D)} \\ \leq C$. In particular, $\widehat \nu_n$ achieves the (near) optimal rate of convergence over Besov balls simultaneously for all $(s,\pi) \in \mathcal A(s_0)$.  
Analogous smoothness adaptive results hold for Theorems~\ref{thm:rateq},~\ref{thm:rateP} and~\ref{prop:rateB} below.

\subsection{Estimation of the density of the mean transition $\Qq$} 

In this section we estimate $(x,y)\leadsto \Qq(x,y)$ for $(x,y) \in \mathcal D^2$  and $\mathcal D$ is a compact interval in $\mathcal S \subseteq \R$.
In a first step, we estimate the density 
$$f_\Qq(x, y) = \nu(x) \Qq(x,y)$$
of the distribution of $(X_{u^-} , X_u)$ when $\mathcal L(X_\emptyset)=\nu$ (a restriction we do not need here) by
$$
\widehat{f}_n(x,y) = \sum_{|\lambda| \leq J } \widehat{f}_{\lambda,n}\psi_\lambda^2(x,y),$$
with
$$\widehat{f}_{\lambda,n} = \mathcal T_{\lambda,\eta}\Big(\frac{1}{|\TT_n^\star|} \sum_{u\in\TT_n^\star} \psi_\lambda^2(X_{u^-},X_u)\Big),
$$
and $\mathcal T_{\lambda,\eta}(\cdot)$ is the hard-threshold estimator defined in Section~\ref{sec:stat_nu} and $\mathbb T_n^\star = \mathbb T_n \setminus \mathbb G_0$.
We can now estimate the density $\Qq(x,y)$ of the mean transition probability by 
\begin{equation} \label{Pphat}
\widehat{\Qq}_n(x,y) = \frac{\widehat{f}_n(x,y)}{\max\{\widehat{\nu}_n(x),\varpi\}}
\end{equation}
for some threshold $\varpi >0$. Thus the estimator $\widehat{\Qq}_n$ is specified by $J$, $\eta$ and $\varpi$.
Define also
\begin{equation} \label{eq:minornu}
m(\nu)= \inf_{x}\nu(x) 
\end{equation}
where the infimum is taken among all $x$ such that $(x,y)\in \mathcal D^2$ for some $y$.

\begin{thm} \label{thm:rateq}
Work under Assumptions~\ref{ass:densityq} and~\ref{ass:ergq} with $\mathfrak n(dx)=dx$. Specify $\widehat \Qq_n$ with 
$$J= \tfrac{1}{2}\log_2\frac{|\mathbb T_n|}{\log |\mathbb T_n|}   \;\; {\rm and} \;\;  \eta = c\sqrt{(\log |\mathbb T_n|)^2/|\mathbb T_n|}$$ for some $c>0$ and $\varpi >0.$
For every $\pi \in [1,\infty], s>2/\pi$ and $p \geq 1$, for  large enough $n$ and $c$ and small enough $\varpi$, the following estimate holds
\begin{equation} \label{upperq}
\Big(\EE \big[ \|  \widehat{\Qq}_n - \Qq \|_{L^p(\Dd^2)}^p \big] \Big)^{1/p}\lesssim \Big(\frac{(\log |\mathbb T_n|)^2}{|\mathbb T_n|}\Big)^{\alpha_2(s,p,\pi)},
\end{equation}
with $\alpha_2(s,p,\pi) = \min\big\{ \tfrac{s}{2s+2} , \tfrac{s/2+1/p-1/\pi}{s+1-2/\pi}\big\}$, provided $m(\nu) \geq \varpi >0$ and
up to a constant that depends on $s,p,\pi, \|\Qq\|_{\Bb_{\pi,\infty}^{s}(\mathcal D^2)}$, $m(\nu)$ and that is continuous in its arguments.\\

This rate is moreover (nearly) optimal: define $\varepsilon_2 =  s\pi- (p-\pi) $. We have
$$
\inf_{ \widehat{\Qq}_n} \sup_{\Qq} \Big(\EE \big[ \| \widehat{\Qq}_n - \Qq \|^p_{L^p(\mathcal D^2)} \big]\Big)^{1/p} \gtrsim 
\left\{
\begin{array}{lll}
|\TT_n|^{-\alpha_2(s,p,\pi) } & \text{if} & \varepsilon_2 > 0 \\
\displaystyle \Big(\frac{\log|\mathbb T_n|}{|\mathbb T_n|}\Big)^{\alpha_2(s,p,\pi)  }& \text{if} & \varepsilon_2 \leq 0, \\
\end{array}
\right.
$$
where the infimum is taken among all estimators of $\Qq$ based on $(X_u)_{u\in\TT_n}$ and the supremum is taken among all $\Qq$ such that  $\|\Qq\|_{\Bb_{\pi,\infty}^{s}(\mathcal D^2)} \leq C$ and $m(\nu)\geq C'$ for some $C,C'>0$.
\end{thm}



\subsection{Estimation of the density of the $\TT$-transition $\Ttransition$} \label{sec:stat_P}


In this section we estimate $(x,y,z)\leadsto \Ttransition(x,y,z)$ for $(x,y,z) \in \mathcal D^3$  and $\mathcal D$ is a compact interval in $\mathcal S \subseteq \R$.
In a first step, we estimate the density 
$$f_\Ttransition(x, y,z) = \nu(x) \Ttransition(x,y,z)$$ of the distribution of $(X_u,X_{u0},X_{u1})$  (when $\mathcal L(X_\emptyset)=\nu$) by
$$
\widehat{f}_n(x,y,z) = \sum_{|\lambda| \leq J } \widehat{f}_{\lambda,n}\psi_\lambda^3(x,y,z),$$
with
$$\widehat{f}_{\lambda,n} = \mathcal T_{\lambda,\eta}\Big(\frac{1}{|\TT_{n-1}|} \sum_{u\in\TT_{n-1}} \psi_\lambda^3(X_u, X_{u0},X_{u1})\Big),
$$
and $\mathcal T_{\lambda,\eta}(\cdot)$ is the hard-threshold estimator defined in Section~\ref{sec:stat_nu}.
In the same way as in the previous section, we can next estimate the density $\Ttransition$ of the $\mathbb T$-transition by
\begin{equation} \label{Pphat}
\widehat{\Ttransition}_n(x,y,z) = \frac{\widehat{f}_n(x,y,z)}{\max\{\widehat{\nu}_n(x), \varpi\}}
\end{equation}
for some threshold $\varpi >0$. Thus the estimator $\widehat{\Ttransition}_n$ is specified by $J$, $\eta$ and $\varpi$.

\begin{thm} \label{thm:rateP}
Work under Assumptions~\ref{ass:densityq},~\ref{ass:ergq} and~\ref{densityTtrans}. Specify $\widehat \Ttransition_n$ with $$J= \tfrac{1}{3}\log_2\frac{|\mathbb T_n|}{\log |\mathbb T_n|}\;\; {\rm and} \;\; \eta = c\sqrt{(\log |\mathbb T_n|)^2/|\mathbb T_n|}$$ for some $c>0$ and $\varpi >0$.
For every  $\pi \in [1,\infty], s>3/\pi$ and $p \geq 1$, for large enough $n$ and $c$ and small enough $\varpi$, the following estimate holds
\begin{equation} \label{upperq}
\Big(\EE \big[ \|  \widehat{\Ttransition}_n - \Ttransition \|_{L^p(\Dd^3)}^p \big]\Big)^{1/p} \lesssim \Big(\frac{(\log |\mathbb T_n|)^2}{|\mathbb T_n|}\Big)^{\alpha_3(s,p,\pi)},
\end{equation}
with $\alpha_3(s,p,\pi) = \min\big\{ \tfrac{s}{2s+3} , \tfrac{s/3+1/p-1/\pi}{2s/3+1-2/\pi}\big\}$, provided $m(\nu) \geq \varpi >0$ and
up to a constant that depends on $s,p,\pi, \|\Ttransition\|_{\Bb_{\pi,\infty}^{s}(\mathcal D^3)}$ and $m(\nu)$ and that is continuous in its arguments.\\

This rate is moreover (nearly) optimal: define $\varepsilon_3 =  \frac{s\pi}{3}- \frac{p-\pi}{2} $. We have
$$
\inf_{ \widehat{\Ttransition}_n} \sup_{\Ttransition} \Big(\EE \big[ \| \widehat{\Ttransition}_n - \Ttransition \|^p_{L^p(\mathcal D^3)} \big]\Big)^{1/p} \gtrsim 
\left\{
\begin{array}{lll}
|\TT_n|^{-\alpha_3(s,p,\pi) } & \text{if} & \varepsilon_3 > 0 \\
\displaystyle \Big(\frac{\log|\mathbb T_n|}{|\mathbb T_n|}\Big)^{\alpha_3(s,p,\pi)  }& \text{if} & \varepsilon_3 \leq 0, \\
\end{array}
\right.
$$
where the infimum is taken among all estimators of $\Ttransition$ based on $(X_u)_{u\in\TT_n}$ and the supremum is taken among all $\Ttransition$ such that  $\|\Ttransition\|_{\Bb_{\pi,\infty}^{s}(\mathcal D^3)} \leq C$ and $m(\nu)\geq C'$ for some $C,C'>0$.
\end{thm}


\section{Applications} \label{sec:examples}

\subsection{Estimation of the size-dependent splitting rate in a growth-fragmentation model} \label{application DHKR1}

Recently, Doumic {\it et al.} \cite{DHKR1} have studied the problem of estimating nonparametrically the size-dependent splitting rate in
growth-fragmentation models (see {\it e.g.} the textbook of Perthame \cite{Perthame}). Stochastically, these are piecewise deterministic Marvov processes on trees that model the evolution of a population of cells or bacteria: 
to each node (or cell) $u \in \TT$, we associate as trait $X_u \in \mathcal S \subset (0,\infty)$ the size at birth of the cell $u$. The evolution mechanism is described as follows: each cell grows exponentially with a common rate $\tau > 0$. A cell of size $x$ splits into two newborn cells of size $x/2$ each (thus $X_{u0}=X_{u1}$ here), with a size-dependent splitting rate $B(x)$ for some $B:\mathcal S \rightarrow [0,\infty)$. Two newborn cells start a new life independently of each other.  If $\zeta_u$ denotes the lifetime of the cell $u$, we thus have
\begin{equation} \label{def B}
\PP\big( \zeta_u \in [t,t+dt) \big| \zeta_u\geq t, X_u=x\big) = B\big(x\exp(\tau t)\big)dt
\end{equation}
and
\begin{equation} \label{rel de base}
X_u = \tfrac{1}{2}X_{u^-} \exp(\tau \zeta_{u^-})
\end{equation}
so that \eqref{def B} and \eqref{rel de base} entirely determine the evolution of the population. We are interested in estimating $x \leadsto B(x)$ for $x \in \mathcal D$ where $\mathcal D \subset \mathcal S$ is a given compact interval. The process $(X_u)_{u \in \mathbb T}$ is a bifurcating Markov chain with state space $\mathcal S$ and $\mathbb T$-transition any version of
$$
\Pp_B(x,dy\,dz) = \PP\big(X_{u0} \in dy, X_{u1} \in dz \,| X_{u^-} = x \big).
$$
Moreover, using \eqref{def B} and \eqref{rel de base}, (see for instance the derivation of Equation $(11)$ in \cite{DHKR1}), it is not difficult to check that 
$$\Pp_B(x,dy\,dz) = Q_B(x,dy) \otimes \delta_y(dz)$$
where $\delta_y$ denotes the Dirac mass at $y$ and
\begin{equation} \label{Pp_B}
Q_B (x, dy) = \frac{B(2y)}{\tau y} \exp \Big( - \int_{x/2}^{y} \frac{B(2z)}{\tau z} dz \Big) {\bf 1}_{\{y \geq x/2\}}dy.
\end{equation}
If we assume moreover that $x \leadsto B(x)$ is continuous, then we have Assumption~\ref{ass:densityq} with $\Qq=Q_B$ and $\mathfrak n(dx)=dx$.\\

Now, let $\mathcal S$ be a bounded and open interval in $(0,\infty)$ such that $\sup \mathcal S > 2 \inf \mathcal S$. Pick $r \in \mathcal S$, $0 < L < \tau \log 2$ and introduce the function class 
$$\mathcal C(r,L)=\Big\{B:\mathcal S \rightarrow [0,\infty), \int^{\sup \mathcal S} \frac{B(x)}{x} dx = \infty,\;\; \int_{\inf \mathcal S}^{r} \frac{B(x)}{x} dx \leq L\Big\}.$$
By Theorem 1.3 in Hairer and Mattingly \cite{HairerMattingly} and the explicit representation \eqref{Pp_B} for $Q_B$, one can check that  for every $B \in \mathcal C(r,L)$, we have Assumption~\ref{ass:ergq} with $\Qq = Q_B$. In particular, we comply with the stringent requirement $\rho = \rho_B \leq C(r,L)$ for some $C(r,L) < 1/2$, {\it i.e.} uniformly over $\mathcal C(r,L)$. Finally, we know by Proposition 2 in Doumic {\it et al.} \cite{DHKR1} -- see in particular Equation $(24)$ -- that
 $$
B(x) = \frac{\tau x}{2} \frac{\nu_B(x/2)}{\int_{x/2}^y \nu_B(z) dz},
$$
where $\nu_B$ denotes the unique invariant probability of the transition $\Qq=Q_B$. This yields a strategy for estimating $x \leadsto B(x)$ via an estimator of $x \leadsto \nu_B(x)$. For a given compact interval $\mathcal D \subset \mathcal S$, define
\begin{equation} \label{Bhat}
\widehat{B}_n(x) = \frac{\tau x}{2} \frac{\widehat{\nu}_n(x/2)}{\big( \tfrac{1}{|\TT_n|}\sum_{u \in \TT_n}{\bf 1}_{\{x/2 \leq X_u < x\}} \big) \vee \varpi},
\end{equation} 
where $\widehat{\nu}_n$ is the wavelet thresholding estimator given in Section~\ref{sec:stat_nu} specified by a maximal resolution level $J$ and a threshold $\eta$ and $\varpi >0$.
As a consequence of Theorem~\ref{thm:ratenu} we obtain the following
\begin{thm}
\label{prop:rateB} 

Specify $\widehat B_n$ with $$J= \tfrac{1}{2}\log_2 \frac{|\mathbb T_n|}{\log |\mathbb T_n|} \;\; {\rm and} \;\; \eta = c\sqrt{\log |\mathbb T_n|/|\mathbb T_n|}$$ for some $c>0$. For every $B \in \mathcal C(r,L)$, $s>0, \pi \in (0,\infty]$ and $p \geq 1$, large enough $n$ and $c$ and small enough $\varpi$, the following estimate holds
$$
\Big(\EE \big[ \|  \widehat{B}_n - B \|_{L^p(\Dd)}^p \big]\Big)^{1/p} \lesssim \Big(\frac{\log |\mathbb T_n|}{|\mathbb T_n|}\Big)^{\alpha_1(s,p,\pi)},$$
with $\alpha_1(s,p,\pi) = \min\big\{ \tfrac{s}{2s+1} , \tfrac{s+1/p-1/\pi}{2s+1-2/\pi}\big\}$,  
up to a constant that depends on $s,p,\pi, \|B\|_{\Bb_{\pi,\infty}^{s}(\mathcal D)}$, $r$ and $L$ and that is continuous in its arguments.\\

This rate is moreover (nearly) optimal: define $\varepsilon_1 =  s\pi- \tfrac{1}{2}(p-\pi) $. We have
$$
\inf_{ \widehat{B}_n} \sup_{B} \Big(\EE \big[ \| \widehat{B}_n - B \|^p_{L^p(\mathcal D)} \big]\Big)^{1/p} \gtrsim 
\left\{
\begin{array}{lll}
|\TT_n|^{-\alpha_1(s,p,\pi) } & \text{if} & \varepsilon_1 > 0 \\
\displaystyle \Big(\frac{\log|\mathbb T_n|}{|\mathbb T_n|}\Big)^{\alpha_1(s,p,\pi)  }& \text{if} & \varepsilon_1 \leq 0, \\
\end{array}
\right.
$$
where the infimum is taken among all estimators of $B$ based on $(X_u)_{u\in\TT_n}$ and the supremum is taken among all $B \in \mathcal C(r,L)$ such that  $\|B\|_{\Bb_{\pi,\infty}^{s}(\mathcal D)} \leq C$. 
\end{thm}

Two remarks are in order: \\

\noindent {\bf 1)} We improve on the results of Doumic {\it et al.} \cite{DHKR1} in two directions: we have smoothness-adaptation (in the sense described in Remark 2) after Theorem~\ref{thm:ratenu} in Section~\ref{sec:stat} for several loss functions over various Besov smoothness classes, while \cite{DHKR1} constructs a non-adapative estimator for H\"older smoothness in squared-error loss; moreover, we prove that the obtained rate is (nearly) optimal in a minimax sense.
 
\noindent {\bf 2)} We unfortunately need to work under the quite stringent restriction that $\mathcal S$ is bounded in order to obtain the uniform ergodicity Assumption~\ref{ass:ergq}, see Remark 3) after Theorem~\ref{thm:triplets} in Section~\ref{sec:deviations}.

\subsection{Bifurcating autoregressive process} \label{application BAR}

Bifurcating autoregressive processes (BAR), first introduced by Cowan and Staudte \cite{CS86}, are yet another stochastic model for understanding cell division. The trait $X_u$ may represent the growth rate of a bacteria $u\in \mathbb T$ in a population of {\it Escherichia Coli} but other choices are obviously possible.
Contrary to the growth-fragmentation model of Section~\ref{application DHKR1} the trait $(X_{u0}, X_{u1})$ of the two newborn cells differ and are linked through the autoregressive dynamics
\begin{equation}\label{eq BAR}
\begin{array}{ll}
\left\{\begin{array}{ll} X_{u0} = f_0(X_{u}) + \sigma_0(X_u) \ep_{u0}, \\
\\ X_{u1} = f_1(X_{u})+ \sigma_1(X_u) \ep_{u1},
\end{array} \right.
\end{array}
\end{equation}
initiated with $X_{\emptyset}$ and
where 
$$f_0, f_1 : \RR \rightarrow \RR\;\;\text{and}\;\;\sigma_0,\sigma_1:\RR \rightarrow (0,\infty)$$
are functions and $(\varepsilon_{u_0}, \varepsilon_{u_1})_{u\in \TT}$ are i.i.d. noise variables with common density function $G:\RR^2 \rightarrow [0,\infty)$  that specify the model.\\

The process $(X_{u})_{u\in\TT}$ is a bifurcating Markov chain with state space $\mathcal S=\R$ and $\TT$-transition
\begin{equation} \label{eq:BARtrans}
\Ttransition(x,dy\,dz) =  G\Big( \sigma_0(x)^{-1} \big(y-f_0(x) \big) , \sigma_1(x)^{-1} \big(z-f_1(x)\big) \Big) dy \,dz.
\end{equation}
This model can be seen as an adaptation of nonlinear autoregressive model when the data have a binary tree structure. The original BAR process in \cite{CS86} is defined for linear link functions $f_0$ and $f_1$ with $f_0 = f_1$
Several extensions have been studied from a parametric point of view, see {\it e.g.}  Basawa and Huggins \cite{BH99, BH00} and Basawa and Zhou \cite{BZ04, BZ05}. More recently,  de Saporta {\it et al.} \cite{BdSGP09, dSGPM12} introduces asymmetry and take into account missing data while Blandin \cite{ Blandin13}, Bercu and Blandin \cite{BB13}, and  de Saporta {\it et al.} \cite{dSGPM13} study an extension with random coefficients. Bitseki-Penda and Djellout \cite{BD14} prove deviation inequalities and moderate deviations for estimators of parameters in linear BAR processes.
From a nonparametric point of view, we mention the applications of \cite{BPEBG15} (Section 4) where deviations inequalities are derived for the Nadaraya-Watson type estimators of $f_0$ and $f_1$ with constant and known functions $\sigma_0$ and $\sigma_1$). 
A detailed nonparametric study of these estimators is carried out in Bitseki Penda and Olivier \cite{BO}.\\

We focus here on the nonparametric estimation of the characteristics of the tagged-branch chain $\nu$ and $\Qq$ and on the $\TT$-transition $\Ttransition$, based on the observation of $(X_u)_{u \in \mathbb T_n}$ for some $n \geq 1$. Such an approach can be helpful for the subsequent study of goodness-of-fit tests for instance, when one needs to assess whether the data $(X_u)_{u \in \mathbb T}$ are generated by a model of the form \eqref{eq BAR} or not.\\

We set  $G_0(x) = \int_\mathcal S G(x,y)dy$ and $G_1(y) = \int_\mathcal S G(x,y)dx$
for the marginals of $G$, and define, for any $M>0$,
\begin{equation*} 
\delta(M) =  \min\big\{\inf_{|x|\leq M}G_{0}(x), \inf_{|x|\leq M}G_{1}(x)\big\}.
\end{equation*}

\begin{assumption}
 \label{ass:BAR} 
For some $\ell>0$ and $\underline{\sigma}>0$, we have
$$
\max\big\{ \sup_x| f_0(x)|, \sup_x| f_1(x)|\big \} \leq \ell < \infty$$
and 
$$\min \big\{ \inf_x \sigma_0(x), \inf_x\sigma_1(x) \big\} \geq \underline{\sigma} > 0.
$$
Moreover, $G_0$ and $G_1$ are bounded and there exists $\mu>0$ and $M>\ell/\underline{\sigma}$ such that $\delta\big((\mu+\ell)/\underline{\sigma}\big) > 0$ and $2(M \underline{\sigma} - \ell) \delta(M) >1/2$.

\end{assumption}


%
Using that  $G_0$ and $G_1$ are bounded, and  \eqref{eq:BARtrans}, we readily check that Assumption~\ref{densityTtrans} is satisfied. We also have Assumption~\ref{ass:densityq} with $\mathfrak n(dx)=dx$ and
\begin{equation*} \label{eq:BARq}
\Qq(x,y) = \tfrac{1}{2}\Big(G_{0}\big(y-f_0(x)) + G_{1}\big(y-f_1(x)\big) \Big),
\end{equation*}
Assumption~\ref{ass:BAR} implies Assumption~\ref{ass:ergq} as well, as follows from an straightfroward adaptation of Lemma 25 in Bitseki Penda and Olivier \cite{BO}. Denoting by $\nu$ the invariant probability of $\Qq$ we also have $m(\nu)>0$ with $m(\nu)$ defined by \eqref{eq:minornu},  for every $\Dd \subset [-\mu,\mu]$, see the proof of Lemma 24 in \cite{BO}.
%
As a consequence, the results stated in Theorems~\ref{thm:ratenu},~\ref{thm:rateq} and~\ref{thm:rateP} of Section~\ref{sec:stat} carry over to the setting of BAR processes satisfying Assumption~\ref{ass:BAR}. We thus readily obtain smoothness-adaptive estimators estimators for $\nu, \Qq$ and $\Ttransition$ in this context and these results are new. 


\subsection{Numerical illustration} \label{numerical}

We focus on the growth-fragmentation model and reconstruct its size-dependent splitting rate.
We consider a perturbation of the baseline splitting rate $\widetilde B(x) = x/(5-x)$ over the range $x\in \Ss = (0,5)$ of the form
\begin{equation*} \label{eq:Btrial}
B(x) = \widetilde B(x)  + \mathfrak{c} \, T\big(2^j (x- \tfrac{7}{2})\big)
\end{equation*}
with $(\mathfrak{c},j) = (3,1)$ or $(\mathfrak{c},j) = (9,4)$, and where 
$T(x)=(1+x){\bf 1}_{\{-1\leq x < 0\}}+(1-x){\bf 1}_{\{ 0 \leq x \leq 1 \}}$ is a tent shaped function. 
Thus the trial splitting rate with parameter $(\mathfrak{c},j) = (9,4)$ is more localized around $7/2$ and higher than the one associated with parameter $(\mathfrak{c},j) = (3,1)$. One can easily check that both $\widetilde B$ and $B$ belong to the class $\Cc(r,L)$ for an appropriate choice of $(r,L)$.
For a given $B$, we simulate $M = 100$ Monte Carlo trees up to the generation $n=15$. To do so, we draw the size at birth of the initial cell $X_\emptyset$ uniformly in the interval $[1.25 , 2.25]$, we fix the growth rate $\tau = 2$ and given a size at birth $X_{u} = x$, we pick $X_{u0}$ according to the density $y \leadsto Q_B(x,y)$ defined by \eqref{Pp_B} using a rejection sampling algorithm (with proposition density $y \leadsto Q_{\widetilde B}(x,y)$) and set $X_{u1}  = X_{u0}$.
Figure \ref{Fig0} illustrates quantitatively how fast the decorrelation on the tree occurs.\\

\begin{figure}[h!]
\includegraphics[width=7cm]{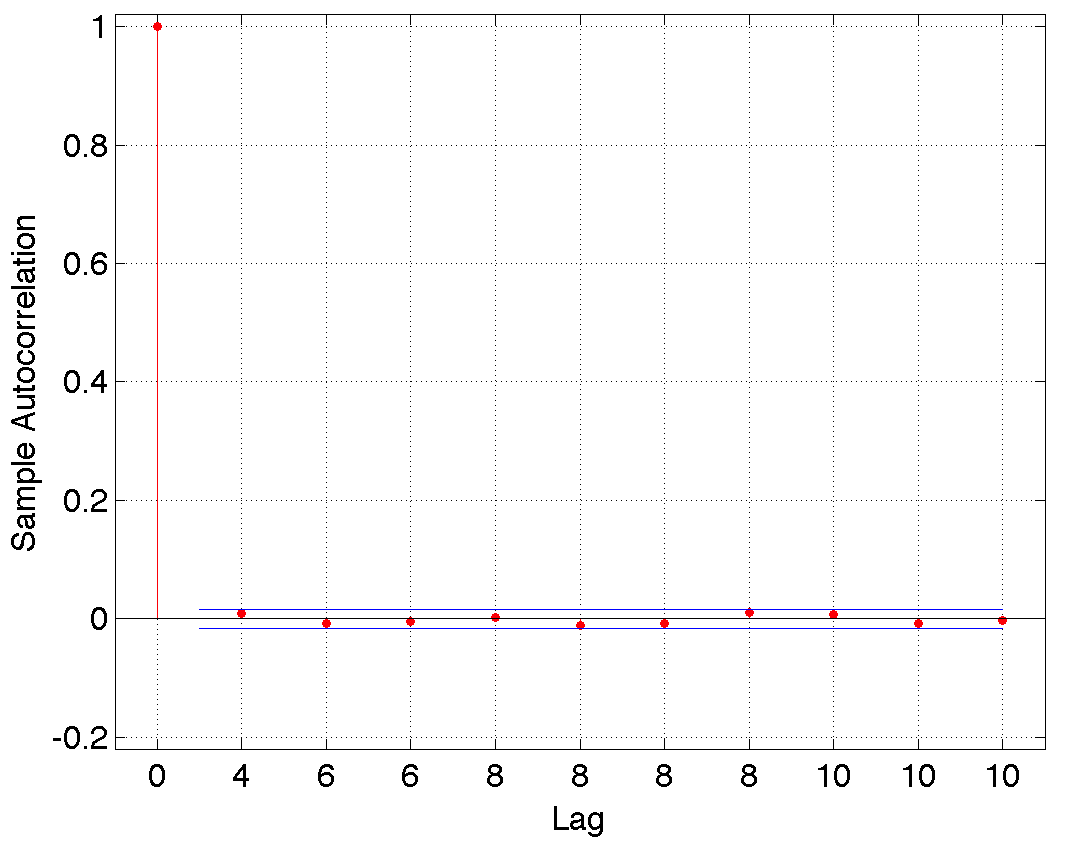}
\caption{{\it Sample autocorrelation of ordered $(X_{u0} , u \in \GG_{n-1})$ for $n=15$. Note: due to the binary tree structure the lags are $\{4, 6, 6,\ldots\}$. As expected, we observe a fast decorrelation.} \label{Fig0}}
\end{figure}

Computational aspects of statistical estimation using wavelets can be found in H\"ardle {\it et al.}, Chapter 12 of \cite{HKPT98}.
We implement the estimator $\widehat B_n$ defined by \eqref{Bhat} using the Matlab wavelet toolbox. 
We take a wavelet filter corresponding to compactly supported Daubechies wavelets of order 8.  As specified in Theorem \ref{prop:rateB}, the maximal resolution level $J$ is chosen as $\tfrac{1}{2}\log_2 (|\mathbb T_n| / \log |\mathbb T_n|)$ and we threshold the coefficients $\widehat \nu_{\lambda,n}$ which are too small by hard thresholding. We choose the threshold proportional to $\sqrt{\log |\TT_n|/|\TT_n|}$ (and we calibrate the constant to  10 or 15 for respectively the two trial splitting rates, mainly by visual inspection). 
%
%
We evaluate  $\widehat B_n$ on a regular grid of $\Dd = [1.5,4.8]$ with mesh $\Delta x = (|\TT_n|)^{-1/2}$. For each sample we compute the empirical error 
$$
e_i = \frac{\|\widehat B_n^{(i)} - B\|_{\Delta x}}{\|B\|_{\Delta x}} , \quad i = 1,\ldots, M,
$$
where $\|\cdot\|_{\Delta x}$ denotes the discrete $L^2$-norm over the numerical sampling and sum up the results through the mean-empirical error $\bar e = M^{-1} \sum_{i = 1}^M e_i$, together with the empirical standard deviation $\big( M^{-1} \sum_{i = 1}^M (e_i-\bar e)^2 \big)^{1/2}$.\\

Table \ref{tab:errorBhat} displays the numerical results we obtained, also giving the compression rate (columns \%) defined as the number of wavelet coefficients put to zero divided by the total number of coefficients. We choose an oracle error as benchmark: the oracle estimator is computed by picking the best resolution level $J^*$ with no coefficient thresholded. We also display the results when constructing $\widehat B_n$ with $\GG_n$ (instead of $\TT_n$), in which case an analog of Theorem \ref{prop:rateB} holds.
For the large spike, the thresholding estimator behaves quite well compared to the oracle for a large spike and achieves the same performance for a high spike.\\
%

%

\begin{table}[h!]
\begin{tabular}{cccclcclcclcc}
\hline \hline 
\multicolumn{1}{l}{}                                                         & \multicolumn{1}{l}{} & \multicolumn{5}{c}{$\boldsymbol{n = 12}$}                                                                                                                                                             &  & \multicolumn{5}{c}{$\boldsymbol{n = 15}$}                                                                                                                                                             \\ \cline{3-7} \cline{9-13} 
\multicolumn{1}{l}{}                                                         & \multicolumn{1}{l}{} & \multicolumn{2}{c}{{\bf Oracle}}                                                                  &  & \multicolumn{2}{c}{{\bf Threshold est.}}                                                       &  & \multicolumn{2}{c}{{\bf Oracle}}                                                                  &  & \multicolumn{2}{c}{{\bf Threshold est.}}                                                       \\ \cline{3-4} \cline{6-7} \cline{9-10} \cline{12-13} 
\multicolumn{1}{l}{}                                                         & \multicolumn{1}{l}{} & \multicolumn{1}{l}{$\boldsymbol{\underset{{\rm (sd.)}}{{\rm Mean}}}$} & \multicolumn{1}{l}{$J^*$} &  & \multicolumn{1}{l}{$\boldsymbol{\underset{{\rm (sd.)}}{{\rm Mean}}}$} & \multicolumn{1}{c}{\%} &  & \multicolumn{1}{l}{$\boldsymbol{\underset{{\rm (sd.)}}{{\rm Mean}}}$} & \multicolumn{1}{l}{$J^*$} &  & \multicolumn{1}{l}{$\boldsymbol{\underset{{\rm (sd.)}}{{\rm Mean}}}$} & \multicolumn{1}{c}{\%} \\ \hline
\multirow{2}{*}{{\bf \begin{tabular}[c]{@{}c@{}}Large\\ spike\end{tabular}}} & $\TT_n$              & $\underset{(0.0159)}{0.0677}$                                         & 5                         &  & $\underset{(0.0196)}{0.1020}$                                         & 96.6                   &  & $\underset{(0.0055)}{0.0324}$                                         & 6                         &  & $\underset{(0.0055)}{0.0502}$                                         & 97.1                   \\
                                                                             & $\GG_n$              & $\underset{(0.0202)}{0.0933}$                                         & 5                         &  & $\underset{(0.0267)}{0.1454}$                                         & 97.9                   &  & $\underset{(0.0081)}{0.0453}$                                         & 6                         &  & $\underset{(0.0097)}{0.0728}$                                         & 96.7                   \\ \\
\multirow{2}{*}{{\bf \begin{tabular}[c]{@{}c@{}}High\\ spike\end{tabular}}}  & $\TT_n$              & $\underset{(0.0180)}{0.1343}$                                         & 7                         &  & $\underset{(0.0163)}{0.1281}$                                         & 97.4                   &  & $\underset{(0.0059)}{0.0586}$                                         & 8                         &  & $\underset{(0.0060)}{0.0596}$                                         & 97.7                   \\
                                                                             & $\GG_n$              & $\underset{(0.0222)}{0.1556}$                                         & 7                         &  & $\underset{(0.0228)}{0.1676}$                                         & 97.7                   &  & $\underset{(0.0079)}{0.0787}$                                         & 8                         &  & $\underset{(0.0087)}{0.0847}$                                         & 97.9                \\  \hline \hline  \\
\end{tabular}
\caption{Mean empirical relative error $\bar e$ and its standard deviation, with respect to $n$, for the trial splitting rate $B$ specified by $(\mathfrak{c},j) = (3,1)$ (large spike) or $(\mathfrak{c},j) = (4,9)$ (high spike) reconstructed over the interval $\Dd  = [1.5,4.8]$ by the estimator $\widehat B_n$. Note: for $n = 15$,  $\tfrac{1}{2}|\TT_n | = 32~767$ and $\tfrac{1}{2}|\GG_n | = 16~384$; for $n = 12$,  $\tfrac{1}{2}|\TT_n | = 4~095$ and $\tfrac{1}{2}|\GG_n | = 2~048$. }
\label{tab:errorBhat}
\end{table}

Figure \ref{Fig1} and Figure \ref{Fig2} show the reconstruction of the size-dependent splitting rate $B$ and the invariant measure $\nu_B$ in the two cases (large or high spike) for one typical sample of size $\tfrac{1}{2}|\TT_n | = 32~767$.  In both cases, the spike is well reconstructed and so are the discontinuities in the derivative of $B$.
As expected, the spike being localized around $\tfrac{7}{2}$ for $B$, we detect it around $\tfrac{7}{4}$ for the invariant measure of the sizes at birth $\nu_B$. The large spike concentrates approximately 50\% of the mass of $\nu_B$ whereas the large only concentrates 20\% of the mass of $\nu_B$.\\

\begin{figure}[h!]
\includegraphics[width=6cm]{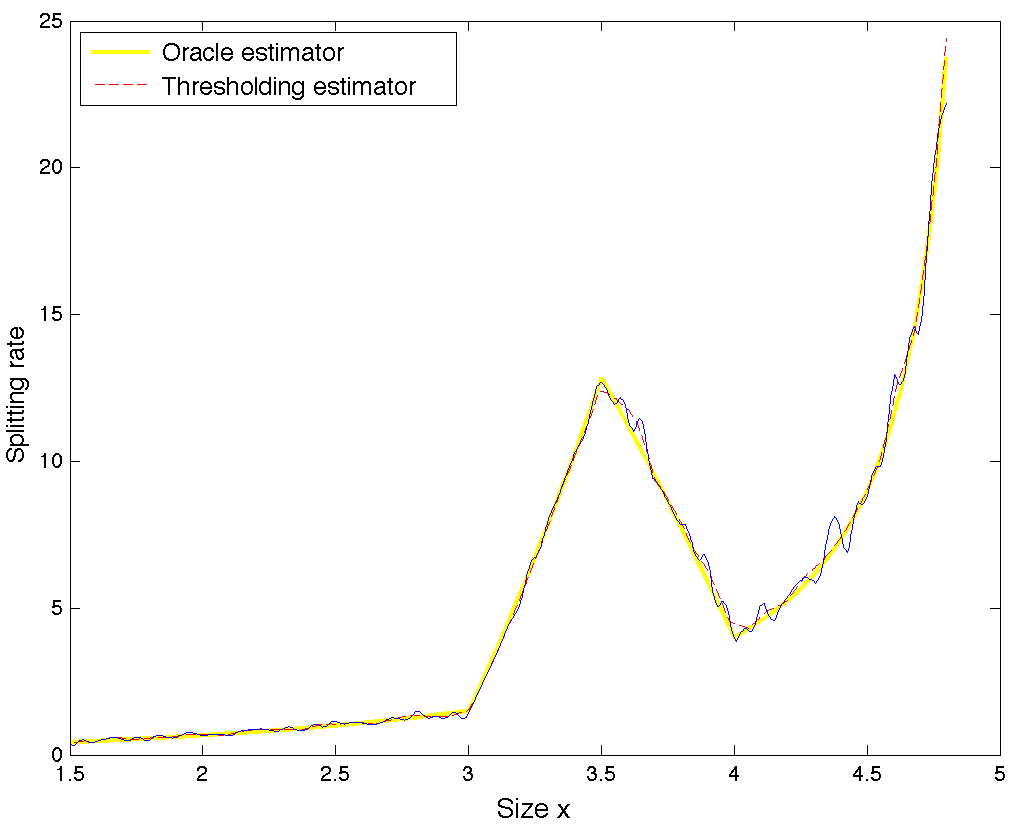}
\includegraphics[width=6cm]{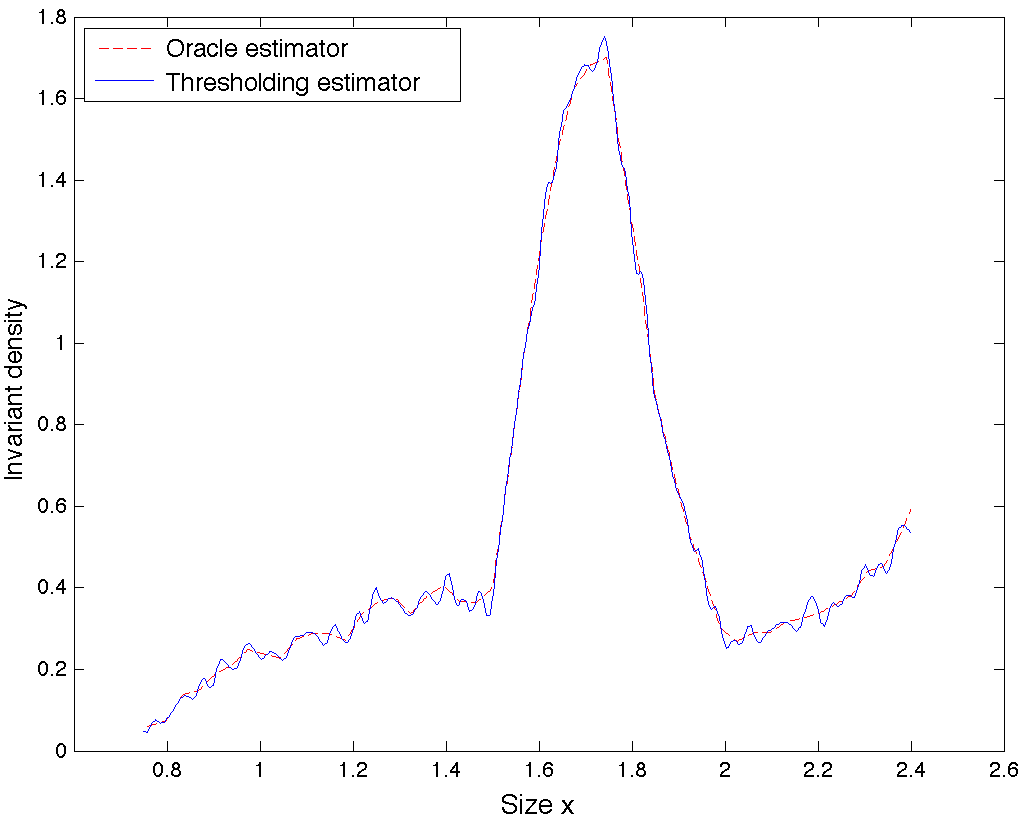}
\caption{{\it Large spike: reconstruction of the trial splitting rate $B$ specified by $(\mathfrak{c},j) = (3,1)$ over $\Dd = [1.5,4.8]$ and reconstruction of $\nu_B$ over $\Dd/2$ based on one sample $(X_u,u \in \TT_n)$ for $n=15$ ({\it i.e.} $\tfrac{1}{2}|\TT_n | = 32~767$).} \label{Fig1}}
\end{figure}

\begin{figure}[h!]
\includegraphics[width=6cm]{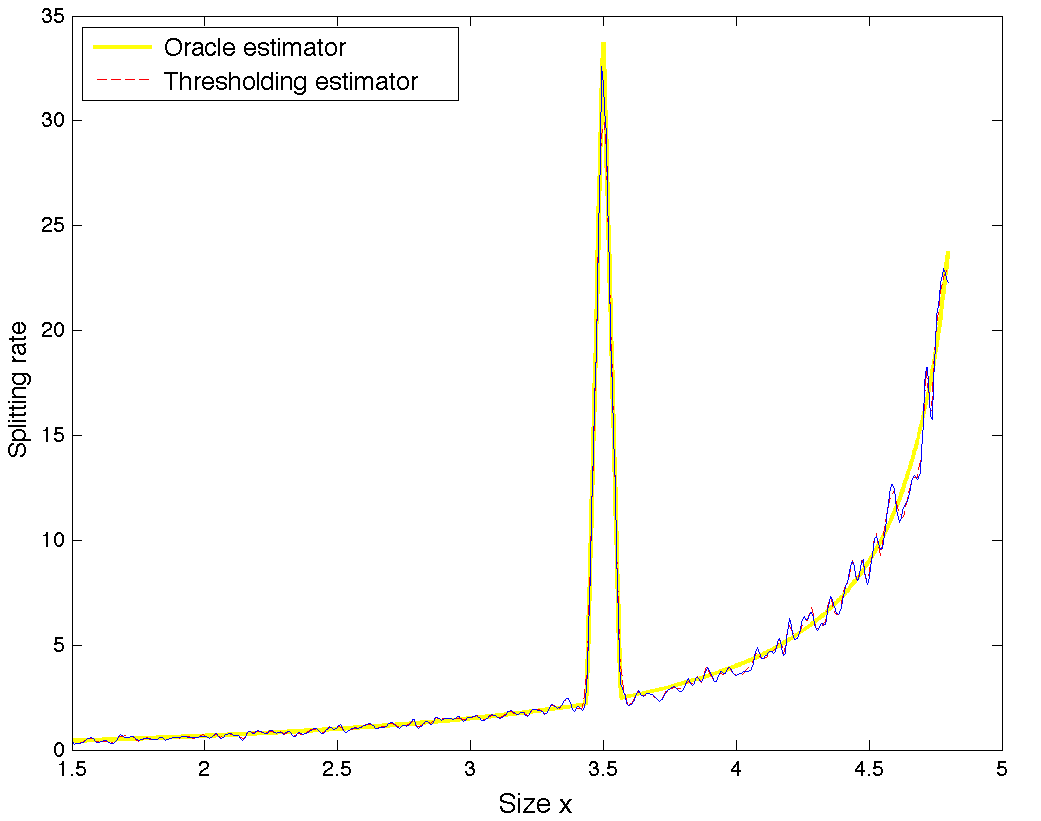}
\includegraphics[width=6cm]{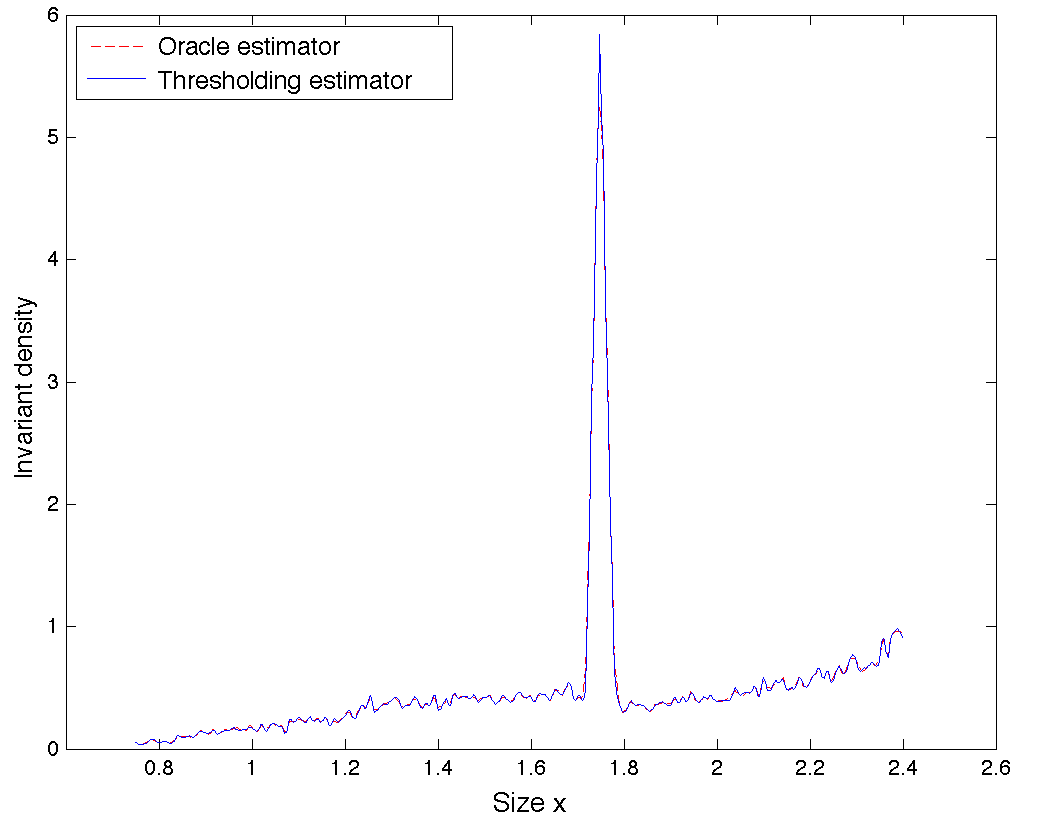}
\caption{{\it High spike : reconstruction of the trial splitting rate $B$ specified by $(\mathfrak{c},j) = (9,4)$ over $\Dd = [1.5,4.8]$ and reconstruction of $\nu_B$ over $\Dd/2$ based on one sample $(X_u,u \in \TT_n)$ for $n=15$ ({\it i.e.} $\tfrac{1}{2}|\TT_n | = 32~767$).} \label{Fig2}}
\end{figure}

\section{Proofs} \label{sec:proofs}

\subsection{Proof of Theorem~\ref{thm:1step}\,(i)} 
Let $g:\mathcal S \rightarrow \R$ such that $|g|_1<\infty$. Set $\nu(g) = \int_{\mathcal S}g(x)\nu(dx)$ and $\widetilde g = g-\nu(g)$.  
Let $n\geq 2$.
By the usual Chernoff bound argument, for every $\lambda > 0$, we have
\begin{equation} \label{chernoff 1}
\PP  \Big( \frac{1}{|\GG_n|} \sum_{u \in \GG_n} \gtilde(X_u) \geq \delta \Big) \leq
\exp \big( -\lambda |\GG_n| \delta \big)  \EE \Big[ \exp \big( \lambda \sum_{u \in \GG_n} \gtilde(X_u) \big) \Big].
\end{equation}
\noindent {\it Step 1}. 
We have
\begin{align*}
 \EE  \Big[  \exp \big( \lambda  \sum_{u \in \GG_n} \gtilde(X_u) \big) \Big| \Ff_{n-1} \Big]  
 & = \EE  \Big[  \prod_{u \in \GG_{n-1}}  \exp \Big( \lambda \big( \gtilde(X_{u0})+\gtilde(X_{u1}) \big) \Big) \Big| \Ff_{n-1} \Big] \\
 & = \prod_{u \in \GG_{n-1}} \EE \Big[ \exp \Big(\lambda \big(\gtilde(X_{u0})+\gtilde(X_{u1}) \big) \Big) \Big| \Ff_{n-1}\Big]
\end{align*}
thanks to the conditional independence of the $(X_{u0},X_{u1})_{u\in\GG_{n-1}}$ given $\Ff_{n-1}$, as follows from Definition~\ref{def BMC}. We rewrite this last term as
$$\prod_{u \in \GG_{n-1}} \hspace{-0.15cm} \EE \Big[ \exp \Big(\lambda \big(\gtilde(X_{u0}) + \gtilde(X_{u1}) - 2 \Qq \gtilde(X_{u})\big) \Big) \big| \Ff_{n-1}\Big] 
\exp \big(\lambda 2 \Qq \gtilde(X_{u}) \big),
$$
inserting the $\mathcal F_{n-1}$-measurable random variable $2\Qq \gtilde(X_{u})$ for $u \in \mathbb G_{n-1}$. Moreover, the bifurcating struture of $(X_u)_{u \in \mathbb T}$ implies
\begin{equation} \label{eq:centered1}
\EE\big[ \gtilde(X_{u0}) + \gtilde(X_{u1}) - 2 \Qq \gtilde(X_{u}) \big| \Ff_{n-1}\big] = 0 , \quad u \in \GG_{n-1},
\end{equation}
since $\Qq=\tfrac{1}{2}(\Pp_0+\Pp_1)$.
%
We will also need the following bound, proof of which is delayed until Appendix
\begin{lem} \label{lem:control1}
Work under Assumptions~\ref{ass:densityq} and~\ref{ass:ergq}. 
For all $r = 0,\ldots, n-1$ and $u \in \GG_{n-r-1}$, we have
\begin{equation*} \label{control11}
\big| 2^r \big(\Qq^{r} \gtilde(X_{u0}) + \Qq^{r} \gtilde(X_{u1}) - 2 \Qq^{r+1} \gtilde(X_{u})\big) \big| \leq c_1 |g|_\infty
\end{equation*}
and
\begin{equation*} \label{control12}
\EE \Big[\Big(2^r \big(\Qq^{r} \gtilde(X_{u0}) + \Qq^{r} \gtilde(X_{u1}) - 2 \Qq^{r+1} \gtilde(X_{u})\big)\Big)^2 \Big| \Ff_{n-r-1} \Big] \leq c_2 \sigma^2_r(g), 
\end{equation*}
with  $c_1 = 4 \max \big\{ 1+R\rho , R(1+\rho)\big\}$, $c_2=  4\max \{ |\Qq|_{\mathcal D}, 4 |\Qq|_{\mathcal D}^2,  4 R^2(1+\rho)^2 \}$ and
\begin{equation} \label{eq:variance}
\sigma^2_r(g) =
\left\{
\begin{array}{ll}
  |g|_2^2 \,  &   r = 0,\\ 
  \min\big\{ |g|_1^2  2^{2r},|g|^2_{\infty} (2 \rho)^{2r} \big\} &  r = 1,\ldots, n-1.
\end{array}
\right.
\end{equation}
(Recall that $|\Qq|_\mathcal D = \sup_{x \in \mathcal S, y \in \mathcal D}\Qq(x,y)$ and $R,\rho$ are defined via Assumption~\ref{ass:ergq}.)
\end{lem}
In view of \eqref{eq:centered1} and Lemma~\ref{lem:control1} for $r=0$, we plan to use the bound
\begin{equation} \label{eq:Bennett}
\EE\big[\exp(\lambda Z) \big] \leq \exp\Big(\frac{\lambda^2 \sigma^2}{2(1 - \lambda M / 3)}\Big)
\end{equation}
valid for any $\lambda\in(0,3/M)$, any  random variable $Z$ such that $| Z | \leq M$, $\EE[Z] = 0$ and $\EE[Z^2] \leq \sigma^2$.
\noindent Thus, for any $\lambda \in \big(0, 3/c_1 |g|_\infty\big)$ and any $u \in \GG_{n-1}$, with $Z= \gtilde(X_{u0}) + \gtilde(X_{u1}) - 2 \Qq \gtilde(X_{u})$, we obtain
$$
\EE \Big[ \exp \Big(\lambda \big( \gtilde(X_{u0}) + \gtilde(X_{u1}) - 2 \Qq \gtilde(X_{u})\big) \Big) \Big| \Ff_{n-1} \Big] \leq \exp \Big(  \frac{ \lambda^2 c_2 \sigma_0^2(g) }{2 (1 - \lambda c_1 |g|_\infty /3 )} \Big).
$$
It follows that
\begin{equation} 
\label{eq:step1}
 \EE  \Big[  \exp \big( \lambda  \sum_{u \in \GG_n} \gtilde(X_u) \big) \Big| \Ff_{n-1} \Big]  \leq  \exp \Big(  \frac{\lambda^2 c_2 \sigma_0^2(g)|\GG_{n-1}| } {2 (1 - \lambda c_1 |g|_\infty /3 )} \Big)  
  \prod_{u \in \GG_{n-1}}  \exp \big(\lambda 2 \Qq \gtilde(X_{u}) \big) .
\end{equation}
%
%
\noindent {\it Step 2}. We iterate the procedure in Step 1. Conditioning with respect to $\Ff_{n-2}$, we need to control 
$$
\EE\Big[\prod_{u \in \GG_{n-1}} \exp \big(\lambda 2 \Qq \gtilde(X_{u}) \big)\Big| \Ff_{n-2} \Big],
$$
and more generally, for $1\leq r \leq n-1$:
\begin{align*}
& \EE\Big[ \prod_{u \in \GG_{n-r}} \exp\big(\lambda 2^{r} \Qq^{r}\gtilde(X_u) \big) \Big| \Ff_{n-r-1} \Big] \\   
= & \prod_{u \in \GG_{n-r-1}} \EE \Big[ \exp \Big(\lambda 2^r \big(\Qq^{r} \gtilde(X_{u0}) + \Qq^{r} \gtilde(X_{u1}) - 2 \Qq^{r+1} \gtilde(X_{u})\big) \Big) \Big| \Ff_{n-r-1} \Big] \\
&\times \exp \big( \lambda 2^{r+1} \Qq^{r+1} \gtilde(X_{u}) \big),
\end{align*}
the last equality being obtained thanks to the conditional independence of the $(X_{u0},X_{u1})_{u\in\GG_{n-r-1}}$ given $\Ff_{n-r-1}$. We plan to use \eqref{eq:Bennett} again: for $u \in \GG_{n-r-1}$, we have 
$$\EE \big[ 2^r \big(\Qq^{r} \gtilde(X_{u0}) + \Qq^{r} \gtilde(X_{u1}) - 2 \Qq^{r+1} \gtilde(X_{u})\big) \big| \Ff_{n-r-1}\big] = 0$$
and the conditional variance given $\Ff_{n-r-1}$ can be controlled using Lemma~\ref{lem:control1}.
Using recursively \eqref{eq:Bennett}, for $r = 1, \ldots, n-1$,
$$
 \EE \Big[ \prod_{u \in \GG_{n-1}} \exp \big(\lambda 2 \Qq \gtilde(X_{u}) \big) \Big| \Ff_{0} \Big] \leq
 \prod_{r = 1}^{n - 1}  \exp\Big( \frac{ \lambda^2 c_2  \sigma^2_r(g) |\GG_{n-r-1}| }{2 \big(1 - \lambda c_1 |g|_{\infty} /3\big)} \Big)  
  \exp \big( \lambda 2^n \Qq^n \gtilde(X_{\emptyset}) \big)
$$
for $\lambda \in  \big(0, 3/c_1 |g|_\infty\big)$. By Assumption~\ref{ass:ergq}, 
$$
\exp \big(\lambda 2^n \Qq^n \gtilde(X_{\emptyset})\big) \leq \exp( \lambda 2^n R (2 |g|_{\infty})  \rho^n) \leq \exp(\lambda 2 R |g|_{\infty})
$$
since $\rho < 1/2$. In conclusion
\begin{equation*}
 \EE \Big[ \prod_{u \in \GG_{n-1}} \exp \big(\lambda 2 \Qq \gtilde(X_{u}) \big) \Big] \leq \exp\Big( \frac{ \lambda^2 c_2   \sum_{r = 1}^{n - 1} \sigma^2_r(g) |\GG_{n-r-1}| }{2 \big(1 - \lambda c_1 |g|_{\infty} /3\big)} \Big)   \exp( \lambda 2 R |g|_{\infty}).
\end{equation*}

\vip
\noindent {\it Step 3}. Let $1 \leq \ell \leq n-1$.
By definition of $\sigma^2_r(g)$ -- recall \eqref{eq:variance} -- and using the fact that $(2\rho)^{2r} \leq 1$, since moreover $|\GG_{n-r-1}| = 2^{n-r-1}$, we successively obtain
\begin{align*}
\sum_{r = 1}^{n-1} \sigma_r^2(g)2^{n-r-1} & \leq  2^{n-1}\big( |g|_1^2\sum_{r=1}^\ell 2^r+|g|^2_\infty \sum_{r=\ell+1}^{n-1}2^{-r}(2\rho)^{2r}\big) \\  
& \leq 2^n\big( |g|_1^2 2^{\ell}+|g|_\infty^22^{-\ell}\big) \\
& \leq |\mathbb G_n| \phi_n(g)
\end{align*}
for an appropriate choice of $\ell$, with 
$\phi_n(g) = \min_{1 \leq \ell \leq n-1}\big(|g|_1^2 2^\ell + |g|_\infty^2 2^{-\ell} \big)$.
It follows that 
\begin{equation} \label{eq:step3}
 \EE \Big[ \prod_{u \in \GG_{n-1}} \exp \big(\lambda 2 \Qq \gtilde(X_{u}) \big) \Big]  \leq  \exp\Big(\frac{ \lambda^2 c_2 |\GG_n| \phi_n(g)}{2 \big(1 - \lambda c_1 |g|_\infty /3  \big)} + \lambda 2 R |g|_\infty \Big).
\end{equation}
\vip
\noindent {\it Step 4}. 
Putting together the estimates \eqref{eq:step1} and \eqref{eq:step3} and coming back to \eqref{chernoff 1}, we obtain
$$
\PP  \Big( \frac{1}{|\GG_n|} \sum_{u \in \GG_n} \gtilde(X_u) \geq \delta \Big) \leq
\exp \Big( -\lambda |\GG_n| \delta +\frac{\lambda^2 c_2  |\GG_n| \Sigma_{1,n}(g) }{2 \big(1 - \lambda c_1 |g|_\infty /3  \big)}+ \lambda 2 R |g|_\infty \Big)
$$
with $\Sigma_{1,n}(g) = |g|_2^2 + \phi_n(g)$ for $n \geq 2$ and $\Sigma_{1,1}(g)=\sigma_0^2(g)=|g|_2^2$.
Since $\delta$ is such that $2 R |g|_\infty \leq |\GG_n| \delta/2$, we obtain
$$
\PP  \Big( \frac{1}{|\GG_n|} \sum_{u \in \GG_n} \gtilde(X_u) \geq \delta \Big) \leq
\exp \Big( -\lambda |\GG_n| \frac{\delta}{2} + \frac{ \lambda^2 c_2  |\GG_n| \Sigma_{1,n}(g)}{2 \big(1 - \lambda c_1 |g|_\infty /3  \big)} \Big).
$$
The admissible choice $\lambda=\delta/\big(\tfrac{2}{3}\delta c_1|g|_\infty+2c_2\Sigma_{1,n}(g)\big)$ yields the result.

\subsection{Proof of Theorem~\ref{thm:1step}\,(ii)} 
\noindent {\it Step 1}. 
Similarly to \eqref{chernoff 1}, we plan to use
\begin{equation} \label{chernoff 2}
\PP  \Big( \frac{1}{|\TT_n|} \sum_{u \in \TT_n} \gtilde(X_u) \geq \delta \Big) \leq
\exp \big( -\lambda |\TT_n| \delta \big)  \EE \Big[ \exp \big( \lambda \sum_{u \in \TT_n} \gtilde(X_u) \big) \Big]
\end{equation}
for a specific choice of $\lambda >0$. We first need to control 
 \begin{align*}
\EE \big[ \exp \big(\lambda \sum_{u \in \TT_n} \gtilde(X_u) \big) \big| \Ff_{n-1} \big] = \prod_{u \in \TT_{n-1}} \hspace{-0.16cm} \exp \big(\lambda \gtilde(X_u) \big) \EE \Big[\exp \big(\lambda \sum_{u \in \GG_n}\gtilde(X_u) \big)  \Big| \Ff_{n-1} \Big].
\end{align*}
Using \eqref{eq:step1} to control $\EE \big[\exp \big(\lambda \sum_{u \in \GG_n}\gtilde(X_u) \big)  \big| \Ff_{n-1} \big]$, we obtain
\begin{multline*}
 \EE  \Big[  \exp \big( \lambda  \sum_{u \in \TT_n} \gtilde(X_u) \big) \Big| \Ff_{n-1} \Big] \\
 \leq  \exp \Big(  \frac{\lambda^2 c_2 \sigma_0^2(g)|\GG_{n-1}| } {2 (1 - \lambda c_1 |g|_\infty /3 )} \Big)  
 \prod_{u \in \GG_{n-1}}  \exp \big(\lambda 2 \Qq \gtilde(X_{u}) \big) 
 \prod_{u \in \TT_{n-1}} \exp \big(\lambda \gtilde(X_u) \big).
\end{multline*}
\vip
\noindent {\it Step 2}. We iterate the procedure. 
At the second step, conditioning w.r.t. $\Ff_{n-2}$, we need to control 
$$
\EE\Big[\prod_{u \in \TT_{n-2}} \exp \big(\lambda \gtilde(X_{u}) \big) \prod_{u \in \GG_{n-1}} \exp \big(\lambda \gtilde(X_{u}) + 2 \lambda \Qq \gtilde(X_{u})\big) \Big| \Ff_{n-2} \Big]
$$
and more generally, at the $(r+1)$-th step (for $1 \leq r \leq n-1$), we need to control
\begin{align*}
\EE\Big[& \prod_{u \in \TT_{n-r-1}} \exp \big(\lambda \gtilde(X_{u}) \big) \prod_{u \in \GG_{n-r}} \exp \big(\lambda \sum_{m = 0}^r 2^m \Qq^m \gtilde(X_{u})\big) \Big| \Ff_{n-r-1} \Big] \\
& = \prod_{u \in \TT_{n-r-2}} \exp \big(\lambda \gtilde(X_{u}) \big)  \prod_{u \in \GG_{n-r-1}} \exp \big( \lambda \sum_{m = 0}^{r+1} 2^m \Qq^{m} \gtilde(X_{u})\big)  \\
& \hspace{4.2cm} \times \EE\big[ \exp\big(\lambda \, \Upsilon_r(X_u,X_{u0},X_{u1}) \big) \big| \Ff_{n-r-1} \big], 
\end{align*}
where we set
\begin{equation*} \label{def g_l}
\Upsilon_r(X_u,X_{u0},X_{u1}) = \sum_{m = 0}^{r} 2^m \big( \Qq^m \gtilde(X_{u0}) + \Qq^m \gtilde(X_{u1}) -  2\Qq^{m+1} \gtilde(X_{u}) \big).
\end{equation*}
This representation successively follows from the $\Ff_{n-r-1}$-measurability of the random variable $\prod_{u \in \TT_{n-r-1}} \exp \big(\lambda \gtilde(X_{u})\big)$, the identity  
$$\prod_{u\in \GG_{n-r}} \exp\big(F(X_u)\big) = \prod_{u\in \GG_{n-r-1}} \exp\big(F(X_{u0})+F(X_{u1})\big),$$ the independence of $(X_{u0}, X_{u1})_{u \in \GG_{n-r-1}}$ conditional on $\Ff_{n-r-1}$ and finally the introduction of the term $2 \sum_{m=0}^r 2^m \Qq^{m+1} \widetilde{g}(X_u)$. 

We have, for $u \in \GG_{n-r-1}$
$$ 
\EE \big[ \Upsilon_r(X_u,X_{u0},X_{u1}) \big| \Ff_{n-r-1}\big] = 0, \quad 
$$
and we prove in Appendix the following bound
\begin{lem} \label{lem:control2}
For any  $r = 1,\ldots, n-1$, $u \in \GG_{n-r-1}$, we have
\begin{equation*} \label{control21}
| \Upsilon_r(X_u,X_{u0},X_{u1}) | \leq c_3 |g|_{\infty}
\end{equation*}
and
\begin{equation*} \label{control22}
\EE \big[ \Upsilon_r(X_u,X_{u0},X_{u1})^2 \big| \Ff_{n-r-1}\big] \leq  c_4 \sigma^2_{r}(g) < \infty
\end{equation*}
where $c_3 = 4R(1+\rho)(1-2\rho)^{-1}$, $c_4 = 12  \max \big\{ |\Qq|_{\mathcal D} , 16 |\Qq|_{\mathcal D}^2, 4 R^2(1+\rho)^2 (1-2\rho)^{-2}  \big\}$ and
\begin{equation} \label{def sigma bis}
\sigma^2_{r}(g) =  |g|_2^2 + \min_{\ell \geq 1}\big(|g|_1^2 2^{2(\ell\wedge r)} + |g|_\infty^2 (2\rho)^{2 \ell} {\bf 1}_{\{ r> \ell\}}\big).
\end{equation}
(Recall that $|\Qq|_\mathcal D = \sup_{x \in \mathcal S, y \in \mathcal D}\Qq(x,y)$ and $R,\rho$ are defined via Assumption~\ref{ass:ergq}.)
\end{lem}
In the same way as for Step 2 in the proof of Theorem~\ref{thm:1step}\,(i), we apply recursively \eqref{eq:Bennett} for $r = 1,\ldots, n-1$ to obtain
$$
\EE \big[ \exp \big(\lambda \sum_{u \in \TT_n} \gtilde(X_u) \big) \big| \Ff_0 \big]  \leq \prod_{r = 0}^{n-1}  \exp \Big(\frac{ c_4  \lambda^2 \sigma_{r}^2(g) |\GG_{n-r-1}| }{2(1 - c'_3 \lambda |g|_{\infty}/3 )} \Big) 
\exp\big( \lambda \sum_{m = 0}^n 2^m \Qq^m \gtilde(X_{\emptyset}) \big),
$$
if $\lambda \in \big( 0, 3/c'_3|g|_{\infty}\big)$ with $c'_3 = \max\{c_1, c_3\} = 4 \max\{ 1+R\rho, R(1+\rho)(1-2\rho)^{-1}\}$ and $\sigma_0^2(g) = |g|_2^2$ in order to include Step~1 (we use $c_4 \geq c_2$ as well). Now, by Assumption~\ref{ass:ergq}, this last term can be bounded by 
$$\exp\big( \lambda \sum_{m = 0}^n 2^m (R |\gtilde|_{\infty} \rho^m)\big) \leq \exp\big(\lambda 2 R (1-2\rho)^{-1} |g|_{\infty} \big)$$
since $\rho < 1/2$.
%
Since $|\GG_{n-r-1}| = 2^{n-r-1}$, by definition of $\sigma_{r}^2(g)$ -- recall \eqref{def sigma bis} -- for any  $1\leq \ell \leq n-1$ and using moreover that  $(2\rho)^\ell\leq 1$, we obtain
\begin{align*}
& \sum_{r = 0}^{n-1} \sigma_r^2(g) |\GG_{n-r-1}|  \\
& \leq 2^{n-1} \bigg( |g|_2^2  \sum_{r = 0}^{n-1} 2^{-r} + |g|_1^2 \Big( \sum_{r = 1}^\ell 2^{2r} 2^{-r} +  \sum_{r = \ell+1}^{n-1} 2^{2\ell} 2^{-r}\Big) + |g|_\infty^2 \sum_{r = \ell + 1}^{n-1} 2^{-r} \bigg) \\
& \leq |\TT_n| \Sigma_{1,n}(g),
\end{align*}
where $\Sigma_{1,n}(g)$ is defined in \eqref{def sigma 1n}.
Thus
$$
\EE \Big[ \exp \big(\lambda \sum_{u \in \TT_n} \gtilde(X_u) \big) \Big]  \leq \exp\Big(\frac{ c_4 \lambda^2 |\TT_n| \Sigma_{1,n}(g) }{2 \big(1 - c'_3 \lambda |g|_\infty /3  \big)} + \lambda 2 R(1-2\rho)^{-1} |g|_\infty \Big).
$$
\vip
\noindent {\it Step 3}. Coming back to \eqref{chernoff 2},
for $\delta>0$ such that $2 R(1-2\rho)^{-1} |g|_\infty \leq |\TT_n| \delta/2$, we obtain
$$
\PP  \Big( \frac{1}{|\TT_n|} \sum_{u \in \TT_n} \gtilde(X_u) \geq \delta \Big) \leq
\exp \Big( -\lambda |\TT_n| \frac{\delta}{2} + \frac{ c_4 \lambda^2 |\TT_n| \Sigma_{1,n}(g)}{2 \big(1 - c'_3 \lambda |g|_\infty /3  \big)} \Big).
$$
We conclude in the same way as in Step 4 of the proof of Theorem~\ref{thm:1step}\,(i).



\subsection{Proof of Theorem~\ref{thm:triplets}\,(i)} 

The strategy of proof is similar as for Theorem~\ref{thm:1step}. Let $g:\Ss^3\rightarrow \R$ such that $|g|_1<\infty$ and set $\gtilde = g-\nu(\Ttransition g)$. Let $n\geq 2$ (if $n=1$, set $\Sigma_{2,1}(g) = |\Qq(\Ttransition g)|_\infty$).
Introduce the notation $\Delta_u = (X_u,X_{u0},X_{u1})$ for simplicity. For every $\lambda >0$, the usual Chernoff bound reads
\begin{equation} \label{chernoff 3}
\PP \Big( \frac{1}{|\GG_n|} \sum_{u \in \GG_n} \gtilde(\Delta_u) \geq \delta \Big) \\ \leq \exp( -\lambda |\GG_n| \delta) \EE \Big[ \exp \big( \lambda \sum_{u \in \GG_n} \gtilde(\Delta_u)\big) \Big].
\end{equation}

\noindent {\it Step 1}. 
We first need to control 
\begin{align*}
 \EE \big[ \exp \big( \lambda \sum_{u \in \GG_n} \gtilde(\Delta_u) \big)\big| \Ff_{n-1} \big] & = \EE \Big[ \prod_{u \in \GG_{n-1}}   \exp \Big(\lambda \big(\gtilde(\Delta_{u0}) + \gtilde(\Delta_{u1}) \big) \Big)  \Big| \Ff_{n-1}\Big] \\
& =  \prod_{u \in \GG_{n-1}}  \EE \Big[ \exp \Big(\lambda \big(\gtilde(\Delta_{u0}) + \gtilde(\Delta_{u1}) \big) \Big)   \Big| \Ff_{n-1}\Big]
\end{align*}
using the conditional independence of the $(\Delta_{u0},\Delta_{u1})$ for  $u \in \GG_{n-1}$  given $\mathcal F_{n-1}$.
%
Inserting the term $2 \Qq (\Ttransition \gtilde)(X_{u})$, this last quantity ia also equal to 
%
$$
\prod_{u \in \GG_{n-1}}  \EE \Big[  \exp \Big(\lambda \big( \gtilde(\Delta_{u0}) + \gtilde(\Delta_{u1}) - 2\Qq(\Ttransition \gtilde)(X_u) \big) \Big)  \Big| \Ff_{n-1}\Big] 
 \exp \big(\lambda 2\Qq(\Ttransition \gtilde)(X_u) \big).
$$
For $u \in \GG_{n-1}$ we successively have
$$\EE \big[ \gtilde(\Delta_{u0}) + \gtilde(\Delta_{u1}) - 2\Qq(\Ttransition \gtilde)(X_u) \big| \Ff_{n-1} \big] = 0,$$
$$| \gtilde(\Delta_{u0}) + \gtilde(\Delta_{u1}) - 2\Qq(\Ttransition \gtilde)(X_u) | \leq 4 (1+R\rho) |g|_{\infty}$$
and
$$\EE \big[\big( \gtilde(\Delta_{u0}) + \gtilde(\Delta_{u1}) - 2\Qq(\Ttransition \gtilde)(X_u) \big)^2 \big| \Ff_{n-1}\big] \leq 4 |\Qq|_{\mathcal D}|\Ttransition g^2|_1,$$
with $|\Qq|_\mathcal D = \sup_{x \in \mathcal S, y \in \mathcal D}\Qq(x,y)$ and $R,\rho$ defined via Assumption~\ref{ass:ergq}. The first equality is obtained by conditioning first on $\mathcal F_n$ then on $\mathcal F_{n-1}$. The last two estimates are obtained in the same line as the proof of Lemma~\ref{lem:control1} for $r=0$, using in particular 
$\Qq(\Ttransition g^2)(x)=\int_\mathcal S \Ttransition g^2(y)\Qq(x,y)\mathfrak{n}(dy) \leq |\Qq|_\mathcal D|\Ttransition g^2|_1$
since $\Ttransition g^2$ vanishes outside $\mathcal D$. 

Finally, thanks to 
\eqref{eq:Bennett} with $Z=\gtilde(\Delta_{u0}) + \gtilde(\Delta_{u1}) - 2\Qq(\Ttransition \gtilde)(X_u)$, we infer
\begin{equation} \label{eq:2step1}
 \EE \big[ \exp \big( \lambda \sum_{u \in \GG_n} \gtilde(\Delta_u) \big)\big| \Ff_{n-1} \big]
 \leq \exp\Big( \frac{\lambda^2 4 |\Qq|_{\mathcal D}|\Ttransition g^2|_1}{2(1- \lambda 4 (1+R\rho) |g|_\infty /3)}  \Big) 
  \prod_{u \in \GG_{n-1}}  \exp \big(\lambda 2\Qq(\Ttransition \gtilde)(X_u) \big)
\end{equation}
for $\lambda \in \big(0,3/(4 (1+R\rho) |g|_\infty)\big)$.\\

\noindent {\it Step 2}.  We wish to control
$
\EE\big[  \prod_{u \in \GG_{n-1}}  \exp \big(\lambda 2\Qq(\Ttransition \gtilde)(X_u) \big)\big].
$
We are back to Step 2 and Step 3 of the proof of Theorem~\ref{thm:1step}\,(i), replacing $\gtilde$ by $\Ttransition \gtilde$, which satisfies $\nu(\Ttransition \gtilde) = 0$.
Equation \eqref{eq:step3} entails
\begin{equation}  \label{eq:2step3}
 \EE \Big[ \prod_{u \in \GG_{n-1}} \hspace{-0.15cm} \exp \big(\lambda 2 \Qq (\Ttransition \gtilde)(X_{u}) \big) \Big]  \leq  \exp\Big(\frac{ \lambda^2 c_2 |\GG_n| \phi_n(\Ttransition g) }{2 \big(1 - \lambda c_1 |\Ttransition g|_\infty /3  \big)} + \lambda 2 R |\Ttransition g|_\infty \Big)
\end{equation}
with 
$\phi_n(\Ttransition g) = \min_{1 \leq \ell \leq n-1}\big(|\Ttransition g|_1^2 2^\ell + |\Ttransition g|_\infty^2 2^{-\ell} \big)$ and $c_1 = 4 \max \big\{ 1+R\rho, R(1+\rho)\big\}$, $c_2 = 4\max \{ |\Qq|_{\mathcal D },  4 |\Qq|_{\mathcal D}^2,  4 R^2(1+\rho)^2 \}$.\\

\noindent {\it Step 3}. Putting together \eqref{eq:2step1} and \eqref{eq:2step3}, we obtain
\begin{equation}  \label{eq:utilewholetree}
 \EE \Big[ \exp \big( \lambda \sum_{u \in \GG_n} \gtilde(\Delta_u) \big)\Big] \leq \exp \Big( \frac{\lambda^2 c_2  |\GG_n| \Sigma_{2,n}(g)}{2(1-\lambda c_1 | g|_\infty /3)} + \lambda 2 R |\Ttransition g|_\infty \Big) 
\end{equation}
with $\Sigma_{2,n}(g) = |\Ttransition g^2|_1 + \phi_n(\Ttransition g)$
and using moreover $ |g|_{\infty} \geq |\Ttransition g|_\infty$ and $c_1 \geq 4(1+R\rho)$.
Back to \eqref{chernoff 3}, since $2 R |\Ttransition g|_\infty \leq |\GG_n|\delta/2$ we finally infer
$$
\PP \Big( \frac{1}{|\GG_n|} \sum_{u \in \GG_n} g(\Delta_u) - \nu(\Ttransition g) \geq \delta \Big) \leq  \exp\Big( -\lambda |\GG_n| \frac{\delta}{2} +\frac{\lambda^2 c_2  |\GG_n| \Sigma_{2,n}(g)}{2(1-\lambda c_1 | g|_\infty/3)} \Big).
$$
We conclude in the same way as in Step 4 of the proof of Theorem~\ref{thm:1step}\,(i).

\subsection{Proof of Theorem~\ref{thm:triplets}\,(ii)} 
In the same way as before, for every $\lambda > 0$,
\begin{equation} \label{chernoff 4}
\PP \Big( \frac{1}{| \TT_{n-1}|} \sum_{u \in \TT_{n-1}} \gtilde(\Delta_u) \geq \delta \Big) \leq e^{-\lambda |\TT_{n-1}| \delta}\; \EE \Big[ \exp \big( \lambda \sum_{u \in \TT_{n-1}} \gtilde(\Delta_u) \big) \Big].
\end{equation}
Introduce $\Sigma_{2,0}'(g) =  |\Ttransition g^2|_1$ and
$$\Sigma_{2,n}'(g) =  |\Ttransition g^2|_1+\inf_{\ell \geq 1}\big(|\Ttransition g|_1^22^{\ell \wedge (n-1)}+|\Ttransition g|_\infty^22^{-\ell}{\bf 1}_{\{\ell < n-1\}}\big),\;\;\text{for}\;\;n\geq 1.$$ 
It is not difficult to check that \eqref{eq:utilewholetree} is still valid when replacing $\Sigma_{2,n}$ by $\Sigma_{2,n}'$.
We plan to successively expand the sum over the whole tree $\mathbb T_{n-1}$ into sums over each generation $\mathbb G_m$ for $m=0,\ldots, n-1$, apply H\"older inequality, apply inequality \eqref{eq:utilewholetree} repeatedly (with $\Sigma_{2,m}'$)
together with the bound 
$$\sum_{m=0}^{n-1}  |\GG_m| \Sigma_{2,m}'(g)  \leq |\TT_{n-1}| \Sigma_{2,n-1}(g).$$
We thus obtain
\begin{align*}
 \EE \big[ \exp \big( \lambda \sum_{u \in \TT_{n-1}} \gtilde(\Delta_u) \big) \big] =  &\, \EE\big[ \prod_{m=0}^{n-1} \exp \big( \lambda \sum_{u \in \GG_m} \gtilde(\Delta_u)  \big) \big] \\
\leq &\, \Big(\EE\big[ \exp \big( n\lambda \gtilde(\Delta_\emptyset) \big) \big] \prod_{m=1}^{n-1} \EE\big[ \exp \big( n\lambda \sum_{u \in \GG_m} \gtilde(\Delta_u)  \big) \big]  \Big)^{1/n} \\
 \leq &\,\Big(\exp \big( n\lambda 2 |g|_\infty \big)  \prod_{m=1}^{n-1} \exp\big(\frac{ (n\lambda)^2 c_2  |\GG_m|  \Sigma_{2,m}'(g) }{2 \big(1 - (n\lambda) c_1 | g|_\infty /3  \big)} + (n \lambda) 2 R |\Ttransition g|_\infty \big) \Big)^{1/n}\\
 \leq &\,\exp\Big(\frac{\lambda^2 c_2n |\TT_{n-1}|  \Sigma_{2,n-1}(g)}{2(1- c_1 (n \lambda)  | g|_\infty/3)} + 2 \lambda (n R |\Ttransition g|_\infty +  |g|_\infty)   \Big).
\end{align*}
Coming back to \eqref{chernoff 4} and using $2(n R |\Ttransition g|_\infty + |g|_\infty) \leq |\TT_{n-1}|\delta/2$, we obtain
$$
\PP \Big( \frac{1}{| \TT_{n-1} |} \sum_{u \in \TT_{n-1}} \gtilde(\Delta_u) \geq \delta \Big) \leq \exp\Big(-\lambda |\TT_{n-1}| \frac{\delta}{2} + \frac{\lambda^2 c_2 n |\TT_{n-1}| \Sigma_{2,n-1}(g) }{2(1- (n \lambda)  c_1 | g|_\infty/3)} \Big).
$$
We conclude in the same way as in Step 4 of the proof of Theorem~\ref{thm:1step}\,(i).

\subsection{Proof of Theorem~\ref{thm:ratenu}}



Put $c(n) = (\log |\mathbb T_n|/|\mathbb T_n|)^{1/2}$ and note that the maximal resolution $J=J_n$ is such that $2^{J_n} \sim c(n)^{-2}$. Theorem~\ref{thm:ratenu} is  a consequence of the general theory of wavelet threshold estimators, see Kerkyacharian and Picard \cite{KP}.
We first claim that the following moment bounds and moderate deviation inequalities hold: for every $p \geq 1$,
\begin{equation} \label{moment nu}
\E\big[|\widehat \nu_{\lambda, n}- \nu_\lambda|^p\big] \lesssim c(n)^p\;\;\text{for every}\;\; |\lambda | \leq J_n
\end{equation}
and
\begin{equation} \label{deviation nu}
\PP\big(|\widehat \nu_{\lambda, n}- \nu_\lambda| \geq p \varkappa c(n)\big) \leq c(n)^{2p}\;\;\text{for every}\;\; |\lambda | \leq J_n
\end{equation}
provided $\varkappa>0$ is large enough, see Condition \eqref{condition varkappa} below.
In turn, we have Conditions $(5.1)$ and $(5.2)$ of Theorem $5.1$ of \cite{KP} with $\Lambda_n=J_n$ (with the notation of \cite{KP}). By Corollary 5.1 and Theorem 6.1 of \cite{KP} we obtain Theorem~\ref{thm:ratenu}.\\ 

It remains to prove \eqref{moment nu} and \eqref{deviation nu}.
We plan to apply Theorem~\ref{thm:1step}\,(ii) with $g = \psi_\lambda$ and $\delta = \delta_n = p\varkappa c(n)$. First, we have $|\psi_\lambda^1|_p \leq C_p 2^{|\lambda|(1/2-1/p)}$ for $p=1,2,\infty$ by \eqref{scaling psi}, so one readily checks that for 
$$\varkappa \geq \tfrac{4}{p}R(1-2\rho)^{-1} C_\infty (\log |\mathbb T_n|)^{-1},$$
 the condition $\delta_n \geq 4R (1-2\rho)^{-1} |\psi_\lambda^1|_\infty |\TT_n|^{-1}$ is satisfied, and this is always true for large enough $n$. Furthermore, since  $2^{|\lambda|} \leq 2^{J_n}\leq c(n)^{-2}$ it is not difficult to check that
\begin{align} \label{cont sigma 1}
\Sigma_{1,n}(\psi_\lambda^1) & = |\psi_\lambda^1|_2^2 + \min_{1 \leq \ell \leq n-1}\big(|\psi_\lambda^1|_1^2 2^\ell + |\psi_\lambda^1|_\infty^2 2^{-\ell} \big) \leq C
\end{align}
for some $C>0$ and thus $\kappa_3\Sigma_{1,n}(\psi_\lambda) \leq \kappa_3C=C'$ say.
Also
$\kappa_4|\psi_\lambda^1|_\infty \delta_n \leq \kappa_4 C_\infty 2^{|\lambda|/2}c(n)p \varkappa \leq C'' p \varkappa$, where $C''>0$ does not depend on $n$ since $2^{|\lambda|/2}\leq c(n)^{-1}$. 
Theorem~\ref{thm:1step}\,(ii) yields
$$\PP\big(|\widehat \nu_{\lambda, n}- \nu_\lambda| \geq p \varkappa c(n)\big) \leq 2 \exp\Big(-\frac{|\mathbb T_n|p^2\varkappa^2c(n)^2}{C'+C''p\varkappa}\Big) \leq c(n)^{2p}$$
for $\varkappa$ such that
\begin{equation} \label{condition varkappa}
\varkappa \geq \tfrac{1}{2} C'' + \sqrt{ (C'')^2 + \tfrac{4}{p} C'}
\end{equation}
and large enough $n$. Thus \eqref{deviation nu} is proved. Straightforward computations show that \eqref{moment nu} follows using
$\EE \big[ | \widehat{\nu}_{\lambda,n} - \nu_{\lambda} |^{p} \big] = \int_0^{\infty} p u^{p-1} \PP \big( | \widehat{\nu}_{\lambda,n} - \nu_{\lambda} | \geq u \big) du$
and \eqref{deviation nu} again. The proof of Theorem~\ref{thm:ratenu} is complete.
\subsection{Preparation for the proof of Theorem~\ref{thm:rateq}}
\vip
For $h:\mathcal S^2\rightarrow \R$, define 
$
\displaystyle |h|_{\infty,1} =\int_{\Ss} \sup_{x\in \Ss} |h(x,y)| dy
$. For $n \geq 2$, set also
\begin{equation} \label{eq:sigmag2}
\Sigma_{3,n}(h) =  |h|_2^2 + \min_{1 \leq \ell \leq n-1} \big( |h|_1^2 2^\ell+ |h|_{\infty,1}^2 2^{-\ell} \big).
\end{equation}
Recall that under Assumption ~\ref{ass:ergq} with $\mathfrak n(dx) = dx$, we set
$f_\Qq(x,y)=\nu(x)\Qq(x,y)$. Before proving Theorem~\ref{thm:rateq}, we first need the following preliminary estimate
\begin{lem} \label{two step}
Work under Assumption~\ref{ass:densityq} with $\mathfrak n(dx) = dx$ and Assumption~\ref{ass:ergq}. Let $h:\mathcal D^2\rightarrow \R$ be such that 
$|hf_\Qq|_1<\infty$. For every $n \geq 1$ and
for any $\delta \geq 4 |h|_\infty (Rn+1) |\TT_n^\star|^{-1}$, we have 
$$
\PP \Big( \frac{1}{|\TT_n^\star|} \sum_{u \in \TT_n^\star}  h(X_{u^-}, X_u) - \langle h,f_\Qq \rangle \geq \delta \Big)  \leq \exp\Big(\frac{ - n^{-1} |\TT_n^\star| \delta^2}{\kappa_5 \Sigma_{3,n}(h) + \kappa_2 |h|_\infty \delta} \Big)
$$
where $\TT_n^\star = \TT_n \setminus \{\emptyset\}$
and $\kappa_5 =  \max\{|\Qq|_{\mathcal D}, |\Qq|_{\mathcal D}^2\}  \kappa_1(\Qq,\mathcal D)$.
\end{lem}

\begin{proof}
We plan to apply Theorem~\ref{thm:triplets}\,(ii) to $g(x,x_0,x_1) = \frac{1}{2} \big( h(x,x_0)+ h(x,x_1) \big)$. Since $\Qq = \tfrac{1}{2}(\Pp_0 + \Pp_1)$ we readily have $\Ttransition g(x) = \int_{\mathcal D} h(x,y) \Qq(x,y)  dy$. Moreover, in that case,
$$\frac{1}{|\TT_{n-1}|} \sum_{u \in \TT_{n-1}} g(X_u,X_{u0},X_{u1}) = \frac{1}{|\TT_n^\star|} \sum_{u \in \TT_n^\star} h(X_{u^-},X_{u})$$
and 
$\int_{\mathcal S}\Ttransition g(x)\nu(x)dx=\int_{\mathcal S \times \mathcal D} h(x,y)\Qq(x,y)\nu(x)dxdy=\langle h,f_\Qq\rangle.$
We then simply need to estimate  $\Sigma_{2,n}(g)$ defined by \eqref{eq:sigma2}. It is not difficult to check that the following estimates hold
$$|\Ttransition g|_1^2 \leq |\Qq|_{\mathcal D}^2 |h|_1^2,\;\;|\Ttransition g|_\infty^2 \leq |\Qq|_{\mathcal D}^2 |h|_{\infty,1}^2\;\;\text{and}\;\;|\Ttransition g^2|_1 \leq |\Qq|_{\mathcal D} |h|_2^2$$
since $(\Ttransition g^2) (x) \leq \int_{\mathcal D} h(x,y)^2\Qq(x,y)dy$. Thus
$
\Sigma_{2,n}(g) \leq \max\{|\Qq|_{\mathcal D}, |\Qq|_{\mathcal D}^2\}  \Sigma_{3,n}(h)
$
and the result follows.
%
\end{proof}

\subsection{Proof of Theorem~\ref{thm:rateq}, upper bound} \label{preuve thm 9 upper}

\noindent {\it Step 1.} We proceed as for Theorem~\ref{thm:ratenu}. Putting $c(n) = (n \log |\mathbb T_{n}^\star|/|\mathbb T_{n}^\star|)^{1/2}$ and noting that the maximal resolution $J=J_n$ is such that $2^{d J_n} \sim c(n)^{-2}$ with $d = 2$, 
%
we only have to prove that for every $p \geq 1$,
\begin{equation} \label{momentq}
\E\big[|\widehat f_{\lambda, n}- f_\lambda|^p\big] \lesssim c(n)^p\;\;\text{for every}\;\; |\lambda | \leq J_n
\end{equation}
and
\begin{equation} \label{deviationq}
\PP\big(|\widehat f_{\lambda, n}- f_\lambda| \geq p \varkappa c(n)\big) \leq c(n)^{2p}\;\;\text{for every}\;\; |\lambda | \leq J_n.
\end{equation}

We plan to apply Lemma~\ref{two step} with $h(x,y) = \psi_\lambda^d(x,y)= \psi_\lambda^2(x,y)$ and $\delta = \delta_n = p\varkappa c(n)$. With the notation used in the proof of Theorem~\ref{thm:ratenu} one readily checks that for 
$$\varkappa \geq \tfrac{4}{p} (1-2\rho)^{-1} C_\infty (R n +1) (\log |\mathbb T_n^\star|)^{-1}$$
the condition $\delta_n \geq 4 |\psi_\lambda^2|_\infty (Rn+1)|\TT_n^\star|^{-1}$ is satisfied, and this is always true for large enough $n$ and 
 \begin{equation} \label{conditionvarkappa2}
 \varkappa \geq   \tfrac{4}{p} (1-2\rho)^{-1} C_\infty  (2 R +1) .
 \end{equation}

Furthermore, since $|\psi_\lambda^d|_p \leq C_p2^{d |\lambda|(1/2-1/p)}$ for $p=1,2,\infty$ and $2^{d|\lambda|} \leq 2^{d J_n}\leq c(n)^{-2}$
we can easily check
\begin{align*}
\Sigma_{3,n}(\psi_\lambda^d) & = |\psi_\lambda^d|_2^2 + \min_{1 \leq \ell \leq n-1}\big(|\psi_\lambda^d|_1^2 2^\ell + |\psi_\lambda^d|_{\infty,1}^2 2^{-\ell} \big)  \leq C
\end{align*}
for some $C>0$, and thus
$\kappa_5 \Sigma_{3,n}(g) \leq \kappa_5C=C'$
say. Also,
$\kappa_2 |\psi_\lambda^d|_\infty \delta_n \leq \kappa_2 C_\infty 2^{d|\lambda|/2}c(n)p \varkappa \leq C'' p \varkappa$,
where $C''$ does not depend on $n$. Applying Lemma \ref{two step}, we derive
$$\PP\big(|\widehat f_{\lambda, n}- f_\lambda| \geq p \varkappa c(n)\big) \leq 2 \exp\Big(-\frac{n^{-1}|\mathbb T_{n-1}|p^2\varkappa^2c(n)^2}{C'+C''p\varkappa}\Big) \leq c(n)^{2p}$$
as soon as $\varkappa$ satisfies \eqref{conditionvarkappa2} and \eqref{condition varkappa} (with appropriate changes for $C'$ and $C''$). Thus \eqref{deviationq} is proved and \eqref{momentq} follows likewise. By \cite{KP} (Corollary 5.1 and Theorem 6.1), we obtain 
\begin{equation} \label{upper f}
\EE \Big(\big[ \|  \widehat{f}_n - f_\Qq \|_{L^p(\Dd^2)}^p \big]\Big)^{1/p} \lesssim \Big(\frac{ n \log |\mathbb T_n| }{|\mathbb T_n|}\Big)^{\alpha_2(s,p,\pi)}
\end{equation}
as soon as $\|f_\Qq\|_{B^s_{\pi,\infty}(\mathcal D^2)}$ is finite, as follows from $f_\Qq(x,y)=\Qq(x,y)\nu(x)$ and the fact that $\|\nu\|_{B^s_{\pi,\infty}(\mathcal D)}$ is finite too. The last statement can be readily seen from the representation $\nu(x)=\int_{\mathcal S}\nu(y)\Qq(y,x)dy$ and the definition of Besov spaces in terms of moduli of continuity, see {\it e.g.} Meyer \cite{Meyer} or H\" ardle {\it et al.} \cite{HKPT98}, using moreover that $\pi \geq 1$.\\

\noindent{\it Step 2}. Since $\Qq(x,y)=f_\Qq(x,y)/\nu(x)$ and $\widehat \Qq_n(x,y)=\widehat f_n(x,y)/\max\{\widehat \nu_n(x), \varpi\}$, we readily have
$$|\widehat \Qq_n(x,y) - \Qq(x,y) |^p\lesssim \tfrac{1}{\varpi^{p}}\big(|\widehat f_n(x,y)-f_\Qq(x,y)|^p+\tfrac{|f_\Qq|_\infty^p}{m(\nu)^p}|\max\{\widehat \nu_n(x), \varpi\} - \nu(x)|^p\big),$$
where the supremum for $f_\Qq$ can be restricted over $\mathcal D^2$.
Since $m(\nu)\geq \varpi$, we have $|\max\{\widehat \nu_n(x), \varpi\}-\nu(x)| \leq |\widehat \nu_n(x)-\nu(x)|$ for $x \in \mathcal D$, therefore
$$\|\widehat \Qq_n - \Qq\|_{L^p(\mathcal D^2)}^p\lesssim \tfrac{1}{\varpi^{p}} \big(\|\widehat f_n-f_\Qq\|_{L^p(\mathcal D^2)}^p+\tfrac{|f_\Qq|_\infty^p}{m(\nu)^p}\|\nu-\nu_n\|_{L^p(\mathcal D)}^p\big)$$
holds as well. We conclude by applying successively the estimate  \eqref{upper f} and Theorem~\ref{thm:ratenu}.

\subsection{Proof of Theorem~\ref{thm:rateq}, lower bound}
We only give a brief sketch: the proof follows classical lower bounds techniques, bounding appropriate statistical distances along hypercubes, see \cite{DJKP, HKPT98} and more specifically \cite{Clemencon2, H4, Lacour2} for specific techniques involving Markov chains. We separate the so-called {\it dense} and {\it sparse} case.\\

\noindent {\it The dense case $\varepsilon_2 >0$}. Let $\psi_\lambda:\mathcal D^2\rightarrow \R$ a family of (compactly supported) wavelets adapted to the domain $\mathcal D$ and satisfying Assumption~\ref{AssumptionA}. For $j$ such that $|\TT_n|^{-1/2} \lesssim 2^{-j(s+1)}$, consider the family 
$$\Qq_{\epsilon,j}(x,y)= |\mathcal D^2|^{-1}{\bf 1}_{\Dd^2}(x,y) + \gamma |\TT_n|^{-1/2} \sum_{\lambda \in \Lambda_j} \epsilon_\lambda \psi_{\lambda}^2 (x,y)$$
where $\epsilon \in \{-1, 1\}^{\Lambda_j}$ and $\gamma>0$ is a tuning parameter (independent of $n$). Since $|\psi_\lambda^2|_\infty \leq C_\infty 2^{|\lambda|} = C_\infty 2^j$ and since the number of overlapping terms in the sum is bounded (by some fixed integer $N$), we have
$$ \gamma |\TT_n|^{-1/2} \big|\sum_{\lambda \in \Lambda_j} \epsilon_\lambda \psi_{\lambda}^2 (x,y)\big| \leq \gamma |\TT_n|^{-1/2}N C_\infty 2^j \lesssim \gamma.$$
This term can be made smaller than $|\mathcal D^2|^{-1}$ by picking $\gamma$ sufficiently small. Hence $\Qq_{\epsilon,j}(x,y) \geq 0$ and since $\int \psi_\lambda=0$, the family $\Qq_{\epsilon,j}(x,y)$ are all admissible {\it mean} transitions with common invariant measure $\nu(dx)={\bf 1}_{\mathcal D}(x)dx$ and belong to a common ball in $\mathcal B^s_{\pi,\infty}(\mathcal D^2)$. For $\lambda \in \Lambda_j$, define $T_\lambda: \{-1,1\}^{\Lambda_j}\rightarrow \{-1,1\}^{|\Lambda_j|}$ by $T_\lambda(\epsilon_\lambda)=-\epsilon_\lambda$ and $T_\lambda(\epsilon_\mu)=\epsilon_\mu$ if $\mu \neq \lambda$. The lower bound in the dense case is then a consequence of the following inequality
\begin{equation} \label{tv control}
\limsup_n\max_{\epsilon \in \{-1,1\}^{\Lambda_j}, \lambda \in \Lambda_j}\|\PP_{\epsilon,j}^n-\PP_{T_\lambda(\epsilon),j}^n\|_{TV}<1,
\end{equation}
where $\PP_{\epsilon,j}^n$ is the law of $(X_u)_{u \in \mathbb T_n}$ specified by the $\TT$-transition $\Ttransition_{\epsilon,j}=\Qq_{\epsilon,j} \otimes \Qq_{\epsilon,j}$ and the initial condition $\mathcal L(X_\emptyset)=\nu$. \\

We briefly show how to obtain \eqref{tv control}. By Pinsker's inequality, it is sufficient to prove that 
$\EE_{\epsilon,j}^n \big[\log\frac{d\PP^n_{\epsilon,j}}{d\PP^n_{T_\lambda(\epsilon),j}} \big]$ can be made arbitrarily small uniformly in $n$ (but fixed).
We have
\begin{align*}
\EE_{\epsilon,j}^n \Big[ - \log\frac{d\PP^n_{T_\lambda(\epsilon),j}}{d\PP^n_{\epsilon,j}} \Big] 
&= - \sum_{u \in \TT_n} \EE_{\epsilon,j}^n \Big[ \log \frac{\Pp_{T_\lambda(\epsilon),j}(X_{u} , X_{u0}, X_{u1} )}{\Pp_{\epsilon,j}(X_{u} , X_{u0}, X_{u1} )} \Big] \\
& = - \sum_{u \in \TT_{n+1}^\star} \EE_{\epsilon,j}^n \Big[ \log \frac{\Qq_{T_\lambda(\epsilon),j}(X_{u^-} , X_u )}{\Qq_{\epsilon,j}(X_{u^-} , X_u )} \Big] \notag \\
& = - |\TT_{n+1}^\star| \int_{\mathcal D^2} \log \Big( \frac{\Qq_{T_\lambda(\epsilon),j} (x , y)}{\Qq_{\epsilon,j}(x , y )}\Big) \Qq_{\epsilon,j}(x , y ) \nu(dx) dy\\
& \leq \hphantom{-} |\TT_{n+1}^\star| \int_{\mathcal D^2}\Big( \frac{\Qq_{T_\lambda(\epsilon),j} (x , y)}{\Qq_{\epsilon,j}(x , y )}-1\Big)^2\Qq_{\epsilon,j}(x , y )  \nu(dx) dy
\end{align*}
using $-\log(1+z) \leq z^2 - z$ valid for $z \geq -1/2$ and the fact that $\nu(dx)$ is an invariant measure for both $\Qq_{T_\lambda(\epsilon),j}$ and $\Qq_{\epsilon,j}$.
Noting that 
\begin{equation*}
\Qq_{T_\lambda(\epsilon), j} (x,y) = \Qq_{\epsilon, j} (x,y)  - 2\gamma |\TT_n|^{-1/2} \epsilon_{\lambda} \psi_{\lambda}^2 (x, y),
\end{equation*}
we derive
$$\Big| \frac{\Qq_{T_\lambda(\epsilon),j} (x , y)}{\Qq_{\epsilon,j}(x , y )}-1\Big| \leq  \frac{2 \gamma |\TT_n|^{-1/2}C_\infty 2^j}{ 1 - \gamma |\TT_n|^{-1/2} N C_\infty 2^j} \lesssim \gamma |\mathbb T_n|^{-1/2}$$
hence the squared term within the integral is of order $\gamma^2 |\mathbb T_n|^{-1}$ so that, by picking $\gamma$ sufficiently small, our claim about  $\EE_{\epsilon,j}^n \big[\log\frac{d\PP^n_{\epsilon,j}}{d\PP^n_{T_\lambda(\epsilon),j}} \big]$ is proved and \eqref{tv control} follows.\\

\noindent {\it The sparse case $\epsilon_2 \leq 0$}. We now consider the family
$$\Qq_{\lambda,j}(x,y)= |\mathcal D^2|^{-1}{\bf 1}_{\Dd^2}(x,y) + \gamma \big(\tfrac{\log |\mathbb T_n|}{|\TT_n|}\big)^{1/2}\epsilon_\lambda \psi_{\lambda}^2 (x,y)$$
with $\epsilon_\lambda\in \{-1,+1\}$ and $\lambda \in \Lambda_j$, with $j$ such that $\big(\tfrac{\log |\mathbb T_n|}{|\TT_n|}\big)^{1/2}\lesssim 2^{-j(s+1-2/\pi)}$. The lower bound then follows from the representation
$$
 \log\frac{d\PP^n_{\lambda,j}}{d\PP^n_{\nu}} = {\mathcal U}^n_\lambda - \omega_\lambda \log 2^{j}
 $$
where $\PP^n_{\lambda,j}$ and $\PP_\nu^n$ denote the law of $(X_u)_{u \in \mathbb T_n}$ specified by the $\TT$-transitions $\Qq_{\lambda,j} \otimes \Qq_{\lambda,j}$ and $\nu \otimes \nu$ respectively (and the initial condition $\mathcal L(X_\emptyset)=\nu$); the $\omega$'s are such that $\sup_n\max_{\lambda \in \Lambda_j}\omega_\lambda  <1 $, and $\mathcal U_\lambda^n$ are random variables such that $\PP_{\lambda,j}^n\big( {\mathcal U}^n_\lambda \geq - C_1\big) \geq C_2 >0$ for some $C_1,C_2>0$. We omit the details, see {\it e.g.} \cite{Clemencon2, H4, Lacour2}.

\subsection{Proof of Theorem~\ref{thm:rateP}}

\begin{proof}[Proof of Theorem~\ref{thm:rateP}, upper bound]  We closely follow Theorem~\ref{thm:rateq} with $c(n) = (n \log |\mathbb T_{n-1}|/|\mathbb T_{n-1}|)^{1/2}$ and $J=J_n$ such that $2^{d J_n} \sim c(n)^{-2}$ with $d = 3$ now. With $\delta=\delta_n=p\varkappa c(n)$, for $\varkappa \geq   \tfrac{4}{p} (1-2\rho)^{-1} C_\infty  (2 R +1)$, we have $\delta_n \geq 4 |\psi_\lambda^3|_\infty (Rn+1)|\TT_n^\star|^{-1}$.\\

Furthermore, since $|\psi_\lambda^d|_p \leq C_p2^{d |\lambda|(1/2-1/p)}$ for $p=1,2,\infty$ and $2^{d|\lambda|} \leq 2^{d J_n}\leq c(n)^{-2}$
it is not difficult to check that
\begin{align*}
\Sigma_{2,n}(\psi_\lambda)  \leq \max\big\{|\Ttransition|_{\mathcal D,1}|\Qq|_\mathcal D, |\Ttransition|_{\mathcal D, 1}^2\big\} \Sigma_{1,n}(\psi_\lambda) \leq C
\end{align*}
thanks to Assumption~\ref{densityTtrans} and \eqref{cont sigma 1}, and thus
$\kappa_1 \Sigma_{2,n}(g) \leq \kappa_1C=C'$. We also have
$\kappa_2 |\psi_\lambda^d|_\infty \delta_n \leq \kappa_2 C_\infty 2^{|\lambda|d/2}c(n)p \varkappa \leq C'' p \varkappa$,
where $C''$ does not depend on $n$. Noting that $f_{\lambda} = \langle f_\Ttransition, \psi_\lambda^d\rangle = \int \Ttransition\psi_\lambda^d d\nu$, we apply Theorem~\ref{thm:triplets}\,(ii) to $g=\psi_\lambda$ and derive
$$\PP\big(|\widehat f_{\lambda, n}- f_\lambda| \geq p \varkappa c(n)\big) \leq 2 \exp\Big(-\frac{n^{-1}|\mathbb T_{n-1}|p^2\varkappa^2c(n)^2}{C'+C''p\varkappa}\Big) \leq c(n)^{2p}$$
for every $|\lambda | \leq J_n$ as soon as $\kappa$ is large enough
and  the estimate
$$\EE \Big(\big[ \|  \widehat{f}_n - f_\Ttransition \|_{L^p(\Dd^3)}^p \big]\Big)^{1/p} \lesssim \Big(\frac{ n \log |\mathbb T_n| }{|\mathbb T_n|}\Big)^{\alpha_3(s,p,\pi)}$$
follows thanks to the theory of \cite{KP}. The end of the proof follows Step 2 of the proof of Theorem~\ref{thm:rateq} line by line, substituting $f_\Qq$ by $f_\Ttransition$.
\end{proof}
\begin{proof}[Proof of Theorem~\ref{thm:rateP}, lower bound]  This is a slight modification of the proof of Theorem~\ref{thm:rateq}, lower bound. For the dense case $\epsilon_3 >0$, we consider an hypercube of the form
$$\Ttransition_{\epsilon,j}(x,y,z)= |\mathcal D^3|^{-1}{\bf 1}_{\Dd^3}(x,y,z) + \gamma |\TT_n|^{-1/2} \sum_{\lambda \in \Lambda_j} \epsilon_\lambda \psi_{\lambda}^3 (x,y,z)$$
where $\epsilon \in \{-1, 1\}^{\Lambda_j}$ with $j$ such that $|\mathbb T_n|^{-1/2} \lesssim  2^{-j(s+3/2)}$ and $\gamma >0$ a tuning parameter,
 while for the sparse case $\epsilon_3 \leq 0$, we consider the family 
$$\Ttransition_{\lambda,j}(x,y,z)= |\mathcal D^3|^{-1}{\bf 1}_{\Dd^3}(x,y,z) + \gamma \big(\tfrac{\log |\mathbb T_n|}{|\TT_n|}\big)^{1/2} \epsilon_\lambda \psi_{\lambda}^3 (x,y,z)$$
with $\epsilon_\lambda\in \{-1,+1\}$, $\lambda \in \Lambda_j$, and $j$ such that $\big(\tfrac{\log |\mathbb T_n|}{|\TT_n|}\big)^{1/2}\lesssim 2^{-j(s+3(1/2-1/\pi))}$.
The proof then goes along a classical line.
\end{proof}


\subsection{Proof of Theorem~\ref{prop:rateB}}
\begin{proof}[Proof of Theorem~\ref{prop:rateB}, upper bound] 
Set $\widehat v_n(x)=\tfrac{1}{|\mathbb T_n|}\sum_{u \in \mathbb T_n}{\bf 1}_{\{x/2 \leq X_u\leq x\}}$ and  $v_\nu(x)=\int_{x/2}^x\nu_B(y)dy$. By Propositions 2 and 4 in Doumic {\it et al.} \cite{DHKR1}, one can easily check that $\sup_{x\in \mathcal D}\nu_B(x)<\infty$ and $\inf_{x \in \mathcal D} v_\nu(x)>0$ with some uniformity in $B$ by Lemma 2 and 3 in \cite{DHKR1}. For $x \in \mathcal D$, we have
\begin{align*}
 \big|\widehat B_n(x)-B(x)\big|^p \lesssim &\, \tfrac{1}{\varpi^p}\big|\widehat \nu_n(x)-\nu_B(x)\big|^p+\frac{\sup_{x\in \mathcal D}\nu_B(x)^p}{\inf_{x \in \mathcal D} v_\nu(x)^p}\big|\max\{\widehat v_n(x), \varpi\}-v_\nu(x)\big|^p \\
\lesssim &\, \big|\widehat \nu_n(x)-\nu_B(x)\big|^p + \big|\widehat v_n(x)- v_\nu(x)\big|^p. 
\end{align*}
By Theorem~\ref{thm:1step} (ii) with $g={\bf 1}_{\{x/2 \leq \cdot \leq x\}}$, one readily checks  $$\EE \big[ |\widehat v_n(x)- v_\nu(x)|^{p} \big] = \int_0^{\infty} p u^{p-1} \PP \big( | \widehat v_n(x)-v_\nu(x)| \geq u \big) du \lesssim |\mathbb T_n|^{-p/2}$$
and this term is negligible. Finally, it suffices to note that $\|\nu_B\|_{\mathcal B^s_{\pi,\infty}(\mathcal D)}$ is finite as soon as $\|B\|_{\mathcal B^s_{\pi,\infty}(\mathcal D)}$ is finite. This follows from 
$$\nu_B(x)=\int_{\mathcal S}\nu_B(y)Q_B(y,x) dy= \frac{B(2x)}{\tau x}\int_0^{2x} \nu_B(y)\exp\big(-\int_{y/2}^x\frac{B(2z)}{\tau z}dz\big)dy.$$
We conclude by applying Theorem~\ref{thm:ratenu}.
\end{proof}

\begin{proof}[Proof of Theorem~\ref{prop:rateB}, lower bound] 
This is again a slight modification of the proof of Theorem~\ref{thm:rateq}, lower bound. For the dense case $\epsilon_1 >0$, we consider an hypercube of the form
$$
B_{\epsilon,j}(x) = B_0(x) + \gamma |\TT_n|^{-1/2} \sum_{\lambda_j} \epsilon_k \psi_{\lambda}^1(x)
$$
where $\epsilon \in \{-1, 1\}^{\Lambda_j}$ with $j$ such that $|\mathbb T_n|^{-1/2} \lesssim 2^{-j(s+1/2)}$ and $\gamma >0$ a tuning parameter. By picking $B_0$ and $\gamma$ in an appropriate way, we have that $B_0$ and $B_{\ep,j}$ belong to a common ball in $\mathcal B^s_{\pi,\infty}(\mathcal D)$ and also belong to $\mathcal C(r,L)$. The associated $\mathbb T$-transition $\Ttransition_{B_{\epsilon,j}}$ defined in \eqref{Pp_B} admits as mean transition
$$\Qq_{B_{\epsilon,j}} (x, dy) = \frac{B_{\epsilon,j}(2y)}{\tau y} \exp \Big( - \int_{x/2}^{y} \frac{B_{\epsilon,j}(2z)}{\tau z} dz \Big) {\bf 1}_{\{y \geq x/2\}}dy$$
which has a unique invariant measure $\nu_{B_{\epsilon,j}}$. Establishing \eqref{tv control} is similar to the proof of Theorem~\ref{thm:rateq}, lower bound, using the explicit representation for  $\Qq_{B_{\epsilon,j}}$ with a slight modification due to the fact that the invariant measures $\nu_{B_{\epsilon,j}}$ and $\nu_{B_{T_\lambda(\epsilon),j}}$ do not necessarily coincide. We omit the details. 

For the sparse case $\epsilon_1 \leq 0$, we consider the family 
$$B_{\lambda,j}(x)=B_0(x) + \gamma \big(\tfrac{\log |\mathbb T_n|}{|\TT_n|}\big)^{1/2} \epsilon_\lambda \psi_{\lambda}^1 (x)$$
with $\epsilon_\lambda\in \{-1,+1\}$, $\lambda \in \Lambda_j$, with $j$ such that $\big(\tfrac{\log |\mathbb T_n|}{|\TT_n|}\big)^{1/2}\lesssim 2^{-j(s+1/2-1/\pi)}$. The proof is then similar.
\end{proof}

\section{Appendix} \label{sec:appendix}

\subsection{Proof of Lemma~\ref{lem:control1}}
\noindent {\it The case $r=0$}. 
By Assumption~\ref{ass:ergq},
$$\big| \gtilde(X_{u0}) + \gtilde(X_{u1}) - 2\Qq \gtilde(X_{u}) \big| \leq 2 \big( |\gtilde|_\infty +R  |\gtilde|_\infty  \rho\big) \leq 4(1+R\rho)  |g|_\infty .$$
This proves the first estimate in the case $r=0$. For $u \in \GG_{n-1}$,
\begin{align*}
\EE \big[ &  \big( \gtilde(X_{u0}) + \gtilde(X_{u1}) - 2\Qq \gtilde(X_{u}) \big)^2 | \Ff_{n-1} \big] \\
 & = \EE \big[ \big( g(X_{u0}) + g(X_{u1}) - 2\Qq g(X_{u}) \big)^2 | \Ff_{n-1} \big]  \\
& \leq \EE \big[ \big( g(X_{u0}) + g(X_{u1}) \big)^2 | \Ff_{n-1} \big] \leq 2 \big( \Pp_0 g^2(X_u) + \Pp_1 g^2(X_u) \big) = 4 \Qq g^2(X_u)
\end{align*}
and for $x \in \Ss$, by Assumption~\ref{ass:densityq}, 
$$
\Qq g^2(x) = \int_{\Ss} g(y)^2 \Qq(x,y) \mathfrak n(dy) \leq |\Qq|_{\mathcal D} |g|_2^2 
$$
since $g$ vanishes outside $\mathcal D$.
Thus 
\begin{equation} \label{eq:regal0}
\EE \big[ \big( \gtilde(X_{u0}) + \gtilde(X_{u1}) - 2\Qq \gtilde(X_{u}) \big)^2 | \Ff_{n-1} \big]  \leq 4 |\Qq|_{\mathcal D} |g|_2^2 
\end{equation}
hence the result for $r=0$.
\vip
\noindent {\it The case $r\geq 1$}. 
On the one hand, by Assumption~\ref{ass:ergq},
\begin{align} 
 \big| 2^r \big(\Qq^{r} \gtilde(X_{u0}) + \Qq^{r} \gtilde(X_{u1}) - 2 \Qq^{r+1} \gtilde(X_{u})\big) \big|
\leq &\, 2^r \big(2R |\gtilde|_\infty (\rho^r + \rho^{r+1})\big) \nonumber
 \\ 
 \leq &\, 4 R(1+\rho) |g|_\infty (2\rho)^r.  \label{bal1}
\end{align}
On the other hand, since
 $$| \Qq g(x) | \leq \int_{\Ss} |g(y)| \Qq(x,y) \mathfrak n(dy)  \leq |\Qq|_{\mathcal D} |g|_1,$$
we also have
\begin{align} 
2^r \big| \Qq^{r} \gtilde(X_{u0}) + \Qq^{r} \gtilde(X_{u1}) - 2 \Qq^{r+1} \gtilde(X_{u}) \big| = & \,2^r \big| \Qq^{r} g(X_{u0}) + \Qq^{r} g(X_{u1}) - 2 \Qq^{r+1} g(X_{u}) \big| \nonumber \\
\leq &\, 2^r 4  |\Qq|_{\mathcal D} |g|_1. \label{bal2}
\end{align}
Putting together these two estimates yields the result for the case $r\geq 1$.

\subsection{Proof of Lemma~\ref{lem:control2}}
By Assumption~\ref{ass:ergq},
$$
| \Upsilon_r(X_u,X_{u0},X_{u1}) | \leq 2 \sum_{m=0}^r 2^m R|\gtilde|_\infty \rho^m (1+\rho) \leq 4 R|g|_\infty (1+\rho) (1-2\rho)^{-1}
$$
since $\rho<1/2$. This proves the first bound. For the second bound we balance the estimates \eqref{bal1} and \eqref{bal2} obtained in the proof of Lemma ~\ref{lem:control1}.
Let $\ell \geq 1$. For $u \in \GG_{n-r-1}$, we have
$$
| \Upsilon_r(X_u,X_{u0},X_{u1}) | \leq I + II + III,
$$
with
\begin{align*}
I & = \big| \gtilde(X_{u0}) + \gtilde(X_{u1}) -  \Qq \gtilde(X_{u}) \big|,  \\
II & =  \sum_{m = 1}^{\ell \wedge r }  2^m \big|  \Qq^m \gtilde(X_{u0}) + \Qq^m \gtilde(X_{u1}) -  2 \Qq^{m+1} \gtilde(X_{u}) \big| , \\
III & =  \sum_{m = \ell\wedge r+1}^r  2^m  \big| \Qq^m \gtilde(X_{u0}) + \Qq^m \gtilde(X_{u1}) -  2 \Qq^{m+1} \gtilde(X_{u}) \big|,
\end{align*}
with $III=0$ if $\ell >r$.
For $u \in \GG_{n-r-1}$, by \eqref{eq:regal0}, we successively have 
$$\EE[I^2 | \Ff_{n-r-1}] \leq 4 |\Qq|_{\mathcal D} |g|_2^2,$$
\begin{align*}
II  \leq  4  |\Qq|_{\mathcal D} |g|_1  \sum_{m = 1}^{\ell \wedge r } 2^m   \leq 8 |\Qq|_{\mathcal D} |g|_1 2^{\ell\wedge r}
\end{align*}
by \eqref{bal2}, while for $\ell  \leq  r$,
\begin{align*}
III  \leq 4 R (1+\rho) |g|_\infty  \sum_{m = \ell +1}^r (2 \rho)^{m} 
\leq  4 R (1+\rho) (1-2\rho)^{-1} |g|_\infty (2 \rho)^{\ell+1}
\end{align*}
by \eqref{bal1}. The result follows.\\


\noindent {\bf Acknowledgements.} We are grateful to A. Guillin for helpful discussions. The research of S.V. B.P. is supported by the Hadamard Mathematics Labex of the Fondation Math\'ematique Jacques Hadamard. The research of M.H. is partly supported by the Agence Nationale de la Recherche, (Blanc SIMI 1 2011 project CALIBRATION).


\begin{thebibliography}{99}


\bibitem{AthreyaKang98} K. B. Athreya and H.-J. Kang. {\it Some limit theorems for positive recurrent branching {M}arkov chains: I}. Advances in Applied Probability, 30 (1998), 693--710.

\bibitem{BH99} I. V. Basawa and R. M. Huggins. {\it Extensions of the bifurcating autoregressive model for cell lineage studies}. Journal of Applied Probability, 36 (1999), 1225--1233.
\bibitem{BH00} I. V. Basawa and R. M. Huggins. {\it Inference for the extended bifurcating autoregressive model for cell lineage studies}. Australian \& New Zealand Journal of Statistics, 42 (2000), 423--432.
\bibitem{BZ04} I. V. Basawa and J. Zhou. {\it Non-Gaussian bifurcating models and quasi-likelihood estimation}. Journal of Applied Probability, 41 (2004), 55-64.
\bibitem{BZ05} I. V. Basawa and J. Zhou. {\it Maximum likelihood estimation for a first-order bifurcating autoregressive process with exponential errors}. Journal of Time Series Analysis, 26 (2005), 825--842.




\bibitem{Benjamini94} I. Benjamini, Y. Peres. {\it Markov chains indexed by trees}. The Annals of Probability (1994), 219--243.

\bibitem{BB13} B. Bercu, V. Blandin. {\it A {R}ademacher-{M}enchov approach for randon coefficient bifurcating autoregressive processes}.  Stochastic Processes an their Applications, 125 (2015), 1218-1243.

\bibitem{BdSGP09} B. Bercu, B. de Saporta and A. G\'egout-Petit. {\it Asymptotic analysis for bifurcating autoregressive processes via a martingale approach}. Electronic Journal of Probability, 87 (2009), 2492--2526.

\bibitem{bertoin} J. Bertoin {\it Random fragmentation and coagulation processes}. Cambridge University Press, 2006.


\bibitem{BD14} S. V. Bitseki Penda, H. Djellout. {\it Deviation inequalities and moderate deviations for estimators of parameters in bifurcating autoregressive models}. Annales de l'Institut Henri Poincar\'e, 50 (2014), 806--844.

\bibitem{BDG14} S. V. Bitseki Penda, H. Djellout and A. Guillin. {\it Deviation inequalities, moderate deviations and some limit theorems for bifurcating Markov chains with application}.  The Annals of Applied Probability, 24 (2014), 235--291.

\bibitem{BPEBG15}  S. V. Bitseki Penda, M. Escobar-Bach and A. Guillin. {\it Transportation cost-information and concentration inequalities for bifurcating Markov chains}.  	{\it arXiv:1501.06693}.  

\bibitem{BO}  S. V. Bitseki Penda, A. Olivier. {\it Nonparametric estimation of the autoregressive functions in bifurcating autoregressive models}. {\it arXiv:1506.01842}.
\bibitem{Blandin13} V. Blandin. {\it Asymptotic results for random coefficient bifurcating autoregressive processes}. Statistics 48 (2014), 1202--1232.

\bibitem{Clemencon2} S. Cl{\'e}men\c{c}on. {\it  Adaptive Estimation of the Transition Density of a Regular Markov Chain by Wavelet Methods}. Mathematical Methods of Statistics, 9 (2000), 323--357.      




\bibitem{CS86} R. Cowan and R. G. Staudte. {\it The bifurcating autoregressive model in cell lineage studies}. Biometrics, 42 (1986), 769--783.


\bibitem{Co} A. Cohen. {\it Wavelets in Numerical Analysis}. In: Ciarlet, P.G., Lions, J.L. (eds.) Handbook of Numerical Analysis, vol. VII. Elsevier, Amsterdam, 2000.


\bibitem{dSGPM11} B. de Saporta, A. G\'egout-Petit and L. Marsalle. {\it Parameters estimation for asymmetric bifurcating autoregressive processes with missing data}. Electronic Journal of Statistics, 5 (2011), 1313--1353.
\bibitem{dSGPM12} B. de Saporta, A. G\'egout-Petit and L. Marsalle. {\it Asymmetry tests for Bifurcating Auto-Regressive Processes with missing data}. Statistics \& Probability Letters, 82 (2012), 1439--1444.
\bibitem{dSGPM13} B. de Saporta, A. G\'egout-Petit and L. Marsalle. {\it Random coefficients bifurcating autoregressive processes}. ESAIM PS, in press.  {\it arXiv:1205.3658}.


\bibitem{DKT} R. DeVore, S.V. Konyagin, V.N. Temlyakov. {\it Hyperbolic Wavelet Approximation. Constructive Approximation}. Springer, New York, 1998.

\bibitem{DJ} D.L. Donoho, I.M. Johnstone. {\it Ideal spatial adaptation via wavelet shrinkage}. Biometrika,  81 (1994), 425--455.

\bibitem{DJ2} D.L. Donoho, I.M. Johnstone. {\it Adapting to unknown smoothness via wavelet shrinkage.} Journal of the American Statistical Association, 90 (1995), 1200--1224.

\bibitem{DJKP2} D.L. Donoho, I.M. Johnstone, G. Kerkyacharian, D. Picard. {\it Wavelet shrinkage: Asymptopia?} Journal of the Royal Statistical Society, 57 (1995), 301--369.

\bibitem{DJKP} D.L. Donoho, I.M. Johnstone, G. Kerkyacharian, D. Picard. {\it Density estimation by wavelet thresholding.} The Annals of Statistics, 54 (1996), 508--539.


\bibitem{DHKR1} M. Doumic, M. Hoffmann, N. Krell and L. Robert. {\it Statistical estimation of a growth-fragmentation model observed on a genealogical tree}. Bernoulli, 21 (2015), 1760--1799.

\bibitem{GaoGuillinWu} F. Gao, A. Guillin and L. Wu. {\it Bernstein-type concentration inequalities for symmetric Markov processes}. Theory of Probability and its Applications, 58 (2014), 358--382.

\bibitem{Guyon} J. Guyon. {\it Limit theorems for bifurcating Markov chains. Application to the detection of cellular aging}. The Annals of Applied Probability, 17 (2007), 1538--1569.

\bibitem{HairerMattingly} M. Hairer and J. Mattingly. {\it Yet Another Look at {H}arris' Ergodic Theorem for {M}arkov Chains}. Seminar on Stochastic Analysis, Random Fields and Applications VI, Progress in Probability, 63 (2011), 109--117.

\bibitem{HKPT98} W. H\"ardle, G. Kerkyacharian, D. Picard and A. Tsybakov. {\it Wavelets, Approximation and Statistical Applications}. Lecture Notes in Statistics, Springer, 1998.

\bibitem{H4} M. Hoffmann. {\it Adaptive estimation in diffusion processes}. Stochastic Processes and their Applications, 79 (1999), 135-163.

\bibitem{KP} G. Kerkyacharian and D. Picard. {\it Thresholding algorithms, maxisets and well-concentrated bases}. Test, 9 (2000), 283-344.

\bibitem{Lacour1} C. Lacour. {\it Adaptive estimation of the transition density of a Markov chain}
     Annales de l'Institut Henri Poincaré,  Probabillit\'es et Statistique,  43 (2007), 571--597.
\bibitem{Lacour2} C. Lacour.   {\it Nonparametric estimation of the stationary density and the transition density of a Markov chain}
      Stochastic Process and their Applications 118 (2008), 232--260. 

\bibitem{Meyer} Y. Meyer. {\it Ondelettes et Opérateurs}, vol. 1. Hermann, Paris, 1990.


\bibitem{MT} S. Meyn and R. Tweedie. {\it Markov chains and stochastic stability}. Springer Berlin Heidelberg, 1993.

\bibitem{Perthame} B. Perthame. {\it Transport Equations in Biology}. Frontiers in Mathematics. Basel: Birkhäuser, 2007.

\bibitem{DHKR2}  L. Robert, M. Hoffmann, N. Krell, S. Aymerich, J. Robert and M. Doumic. {\it Division control in Escherichia coli is based on a size-sensing rather than a timing mechanism}. BMC Biology, 02/2014 12(1):17. 



\bibitem{Takacs01} C. Takacs. {\it Strong law of large numbers for branching Markov chains}. Markov Processes and Related Fields, 8 (2001), 107--116.

\end{thebibliography}
\end{document}